\documentclass[preprint,sort&compress,3p,12pt]{elsarticle}
\usepackage{amsmath,amssymb,amsthm}

\newcommand\be{\begin{equation}}
\newcommand\ee{\end{equation}}
\DeclareMathOperator{\im}{Im }
\DeclareMathOperator{\rea}{Re }
\DeclareMathOperator{\Diag}{Diag }

\numberwithin{equation}{section}

\newtheorem{thm}{Theorem}[section]
\newtheorem{lem}[thm]{Lemma}
\newtheorem{cor}[thm]{Corollary}

\newtheorem{rem}[thm]{Remark}
\newtheorem{example}{Example}

\DeclareMathOperator{\spa}{span}

\allowdisplaybreaks[3]
\begin{document}

\begin{frontmatter}

\title{A Degenerate Hopf Bifurcation Theorem in Infinite Dimensions}

\author[scnu]{Hongjing Pan}
\ead{panhj@m.scnu.edu.cn}

\author[sysu]{Ruixiang Xing}
\ead{xingrx@mail.sysu.edu.cn}

\author[sysu]{Zhannan Zhuang\corref{cor}}
\ead{zhuangzhn@m.scnu.edu.cn}

\cortext[cor]{Corresponding author}

\address[scnu]{School of Mathematical Sciences,
South China Normal University, Guangzhou 510631,  China}

\address[sysu]{School of Mathematics, Sun Yat-sen University, Guangzhou 510275, China}


\begin{abstract}
A Hopf bifurcation theorem is established for the abstract evolution equation $\frac{\mathrm{d}x}{\mathrm{d}t}=F(x,\lambda)$ in infinite dimensions under the degeneracy condition $\rea \mu ^{\prime}(\lambda_0)= 0$ and suitable assumptions. 
The stability properties of bifurcating periodic solutions are also derived. Interestingly, it is shown that a transcritical Hopf bifurcation still can occur at $\lambda_0$ although the stability property of the trivial solutions does not change near $\lambda_0$.
 Our results do not require the analyticity of $F$. 
 The main tools are the Lyapunov--Schmidt reduction and a Morse lemma. 
 Applications to a multi-parameter diffusive predator--prey system discover new branches of periodic solutions.
\end{abstract}

%

\begin{keyword}
Hopf Bifurcation \sep Degeneracy Condition  \sep  Principle of Exchange of Stability \sep Morse Lemma \sep  Diffusive Predator--Prey System.



\MSC[2020] 35B32 \sep 35B35 \sep 35K57 \sep 37G15 \sep 58E09
\end{keyword}

\end{frontmatter}


\section{Introduction}
\label{sec:1}
Hopf bifurcation is an important kind of dynamic bifurcations in various bifurcation phenomena and describes the birth of periodic solutions from equilibria in dynamical systems. 
It is a local bifurcation which occurs in the case when a pair of complex conjugate eigenvalues of the linearization crosses the imaginary axis in the complex plane. Hopf bifurcation is also known as Poincar\'{e}--Andronov--Hopf bifurcation, which was originally investigated in the context of finite dimensional dynamical systems generated by ODEs.
Hopf bifurcation in the setting of PDEs had been studied since the nineteen seventies; see e.g. \cite{Sattinger1971,Ruelle1971,Marsden1976,Joseph1972,Iudovich1971,Iooss1972,Henry1981,Fife1974} for some early works.
A classic and more general Hopf bifurcation theorem in infinite dimensions was established by Crandall and Rabinowitz in a seminal paper \cite{Crandall1977}.
We next state it in the form given in an excellent monograph \cite{Kielhoefer2012} by Kielh\"{o}fer.

Let $X$ and $Z$ be real Banach spaces.
Consider the parameter-dependent evolution equation
\be
\frac{\mathrm{d}x}{\mathrm{d}t}=F(x,\lambda),\label{a1}
\ee
where $F: U\times V\rightarrow Z$, $0\in U\subset X$ and $\lambda_0\in V \subset \mathbb{R}$ are open neighborhoods, $X\subset Z$ is continuously embedded. So $\frac{\mathrm{d}x}{\mathrm{d}t}$ is taken be an element of $Z$, and hence the evolution equation is well defined.
To be convenient later, we introduce some conditions following \cite[I.8]{Kielhoefer2012}:
\begin{enumerate}[(F1)]
  \item $F\in C^3(U\times V,Z)$.
  \item $F(0,\lambda)=0$ for $\lambda\in V$.
  \item $ i\kappa_0(\neq 0)$ is an algebraically simple eigenvalue of $A_0\equiv {D}_xF(0,\lambda_0)$ with eigenvector $\varphi_0 \notin R(i\kappa_0I-A_0)$, and $\pm i\kappa_0I-A_0$  are Fredholm operators of index zero.
 \item $\rea \mu ^{\prime}(\lambda_0)\ne 0$.
  \item For all $n\in \mathbb{Z}\backslash\{\pm1\}$, $in\kappa_0$ is not an eigenvalue of $A_0$.
  \item $A_0$ as a mapping in $Z$, with dense domain $D(A_0)=X$, generates an analytic (holomorphic) semigroup $e^{A_0t}\in L(Z,Z)$ for $t \geqslant 0$ that is compact for $t > 0$.
\end{enumerate}
Condition (F4) is referred to as the \emph{nondegeneracy} condition, in which $^{\prime}=\frac{\mathrm{d}}{\mathrm{d}\lambda}$ and the perturbed eigenvalues $\mu(\lambda)$ of ${D}_xF(0,\lambda)$ near $i\kappa_0$ along the trivial solution line are guaranteed by the Implicit Function Theorem (cf. \cite[Proposition I.7.2]{Kielhoefer2012}) and satisfy
\begin{equation}\label{eq:perteigen}
\begin{split}
&{D}_xF(0,\lambda)\varphi(\lambda)=\mu(\lambda)\varphi(\lambda) \; \hbox{ such that }
\mu(\lambda_0)=i\kappa_0  \hbox{ and } \varphi(\lambda_0)=\varphi_0,\\
&\hbox{and }  \mu(\lambda) \text{ are simple and continuously differentiable near } \lambda_0.
\end{split}
\end{equation} 
 (F5) is known as the \emph{nonresonance} condition. 
The eigenvalues of ${D}_xF(0,\lambda)$ are computed with respect to the natural complexifications $X_{c}$ and $Z_c$ of the real Banach spaces $X$ and $Z$ (cf. \cite[p.35]{Kielhoefer2012}). So $\mu$ is an eigenvalue of ${D}_xF(0,\lambda)$ if and only if $\bar{\mu}$ is an eigenvalue of that.
In what follows, we omit the subscript $c$ for simplicity.

The following form of the Hopf bifurcation theorem in infinite dimensions comes from \cite[Theorem I.8.2 and Corollary I.8.3]{Kielhoefer2012}.
\begin{thm}[See Fig.\ref{fig:surface0}]\label{thm1}
Consider problem \eqref{a1} under conditions (F1)--(F6).
Then there exists a continuously differentiable curve $\{(x(r), \lambda(r))\}$ of $2 \pi / \kappa(r)$-periodic solutions of \eqref{a1} through the bifurcation point $(x(0), \lambda(0))=(0, \lambda_{0})$ with $2 \pi / \kappa(0)={2 \pi}/{\kappa_{0}}$ in $\big(C_{2 \pi / \kappa(r)}^{1+\alpha}(\mathbb{R}, Z) \cap C_{2 \pi / \kappa(r)}^{\alpha}(\mathbb{R}, X)\big) \times \mathbb{R}$.
Every other periodic solution of \eqref{a1} in a neighborhood of $\left(0, \lambda_{0}\right)$ is obtained from $(x(r), \lambda(r))$ by a phase shift $S_{\theta} x(r)$.
In particular, $x(-r)=S_{\pi / \kappa(r)} x(r), \kappa(-r)=\kappa(r),$ and $\lambda(-r)=\lambda(r)$ for all $r \in(-\delta, \delta)$.
Moreover, the tangent vector of $\{(x(r), \lambda(r))\}$ with $2 \pi / \kappa(r)$-period at $(0, \lambda_{0})$ is given by $(\dot{x}(0),\dot{\lambda}(0))=(\rea (\varphi_0 e^{i\kappa_0 t}),0)$ with $\dot{\kappa}(0)=0$ (here, $\dot{}=\frac{\mathrm{d}}{\mathrm{d}r}$).
\end{thm}
\begin{figure} 
\centering
\includegraphics[totalheight=1.5in]{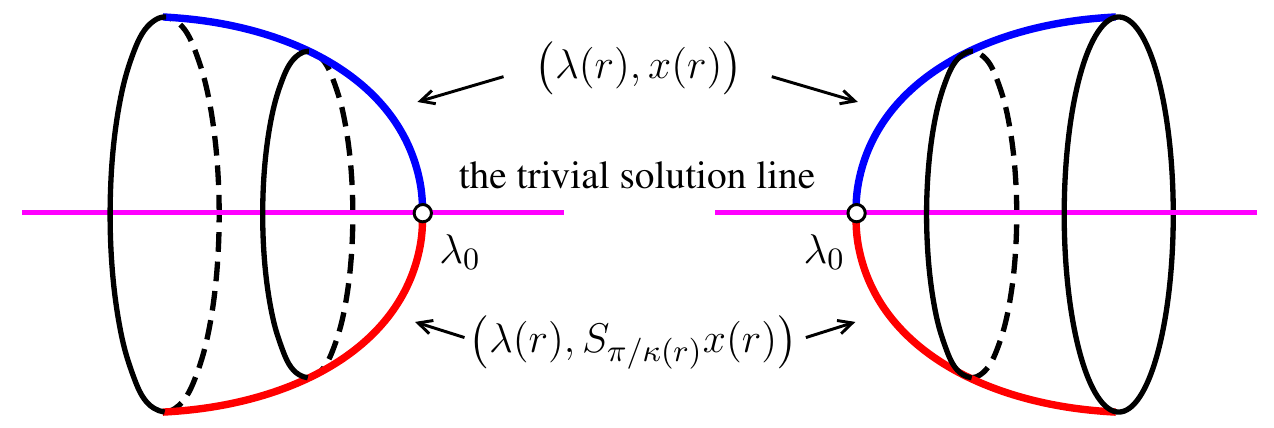}
\caption{Local surfaces composed of bifurcating periodic solutions in Theorem \ref{thm1}. Left: the subcritical case. Right: the supercritical case. 
Every other periodic solution of \eqref{a1} in a neighborhood of $\left(0, \lambda_{0}\right)$ is described from $(x(r), \lambda(r))$ by a phase shift $S_{\theta} x(r)$, 	e.g., the lower branches (red curves) are obtained via the  phase shift $S_{\pi/\kappa(r)}$ of the upper branches (blue curves).}
\label{fig:surface0}
\end{figure}
\begin{rem}
\emph{(a)} In the above theorem, $C_{p}^{\alpha}(\mathbb{R}, X)$ and $C_{p}^{1+\alpha}(\mathbb{R}, Z)$ are the standard H\"{o}lder spaces of $p$-periodic functions. That is,
$C_{p}^\alpha(\mathbb{R},X)=\Big\{z:\mathbb{R}\rightarrow X~|~z(t+p)=z(t),t\in \mathbb{R}, \|z\|_{X,\alpha}=\max\limits_{t\in\mathbb{R}}\|z(t)\|_X+\sup\limits_{s\ne t}\frac{\|z(t)-z(s)\|_X}{|t-s|^\alpha}<\infty\Big\},$
$C_{p}^{1+\alpha}(\mathbb{R},Z)=\Big\{z:\mathbb{R}\rightarrow Z~|~z,\frac{\mathrm{d}z}{ \mathrm{d} t}\in C_{p}^{\alpha}(\mathbb{R},Z),
\|z\|_{Z,1+\alpha}$ $=\|z\|_{Z,\alpha}+\|\frac{\mathrm{d}z}{\mathrm{d} t}\|_{Z,\alpha}<\infty\Big\}$. The phase shift $S_{\theta}$ means that $(S_{\theta} z)(t)=z(t+\theta)$.\\
\emph{(b)} If \eqref{a1} is an ODE, i.e., if $X=Z=\mathbb{R}^n$, then as noted in \cite[Remark I.8.4]{Kielhoefer2012}, the above H\"{o}lder continuity is not needed, and the regularity $F\in C^2(U\times V,\mathbb{R}^n)$ is enough in (F1).
\end{rem}

To state the stability properties of bifurcating solutions, we need the following two conditions:
\begin{enumerate}
 \item[(F1$^{\prime}$)] $F\in C^4(U\times V,Z)$.
 \item[(F7)\,] Apart from the two simple eigenvalues $\pm i \kappa_{0}$, the entire spectrum of $A_{0}$ is in the left complex half-plane.
\end{enumerate}

 The following form of the (linear) stability theorem of the bifurcating periodic solutions  comes from \cite[Theorem I.12.2 and Corollary I.12.3]{Kielhoefer2012}.
\begin{thm}[See Fig.\ref{fig:stable00}]\label{thm:stable1}
 Let $\{(x(r), \lambda(r)) \mid r \in(-\delta, \delta)\}$ be the curve of $2 \pi / \kappa(r)$-periodic solutions of \eqref{a1} according to Theorem \ref{thm1}. Assume that (F1$^{\prime}$) holds.
 Let $\mu_{2}(r)$ be the nontrivial Floquet exponent of $x(r)$ such that $\mu_{2}(0)=0 .$ Then $\dot{\mu}_{2}(0)=0$ and $
\ddot{\mu}_{2}(0)=2 \rea\mu^{\prime}(\lambda_{0}) \ddot{\lambda}(0). $
Further assume that (F7) holds and assume that $\rea  \mu^{\prime}(\lambda_{0})>0,$
i.e., the trivial solution $\{(0, \lambda)\}$ of \eqref{a1} is stable for $\lambda<\lambda_{0}$ and unstable for $\lambda>\lambda_{0}$. Then
$$
\operatorname{sign}(\lambda(r)-\lambda_{0})=\operatorname{sign} \mu_{2}(r) \text { for } r \in(-\delta, \delta),
$$
which means that
the bifurcating periodic solution $\{(x(r), \lambda(r))\}$ of \eqref{a1} is stable, provided that the bifurcation is supercritical, and it is unstable if the
bifurcation is subcritical. If $\rea \mu^{\prime}(\lambda_{0})<0,$ the stability properties of the trivial solution are reversed
and the stability of the bifurcating periodic solution is also reversed.
\end{thm}
\begin{figure} 
\centering
\includegraphics[totalheight=1.20in]{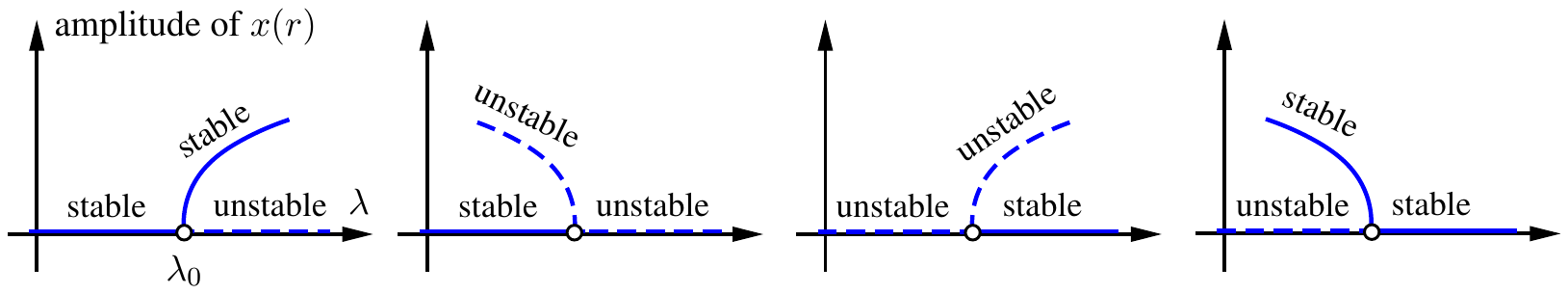}
\caption{The stability properties for the trivial solutions and the bifurcating periodic solution in Theorem \ref{thm:stable1}. The two on the left: $\rea \mu ^{\prime}(\lambda_0)> 0$;
the two on the right: $\rea \mu ^{\prime}(\lambda_0)< 0$.}\label{fig:stable00}
\end{figure}
As sketched in Fig.\ref{fig:stable00}, Theorem \ref{thm:stable1} also implies the \emph{Principle of Exchange of Stability}, which says that consecutive curves of periodic solutions ordered in the $(r, \lambda)$-plane  (including the trivial line $\{(0, \lambda)\} $) have opposite stability properties
 (cf. \cite{Crandall1977a} or \cite[p.89]{Kielhoefer2012}). 

In this paper we are interested in the situation that the nondegeneracy condition (F4) is violated, i.e., the opposite of (F4) holds:
\begin{enumerate}[(F4$^{\prime}$)]
	\item $\rea \mu ^{\prime}(\lambda_0)= 0$.
\end{enumerate}
Condition (F4) plays a crucial role in Theorem \ref{thm1} and means
that the simple eigenvalue $\mu(\lambda)$ of $D_xF(0, \lambda)$ leaves the stable left complex half-plane through the imaginary axis at $\lambda =\lambda_0$ with ``nonvanishing speed'', which leads to the change of stability of the trivial solution.
It is natural to ask what happens if (F4$^{\prime}$) holds (i.e., through the imaginary axis with ``vanishing speed''), and whether a Hopf bifurcation occurs even if the stability of the trivial solution does not change near $\lambda_0$? 

%
For the case of (F4$^{\prime}$),
degenerate Hopf bifurcations in infinite dimensions had been deeply studied by Kielh\"{o}fer \cite{Kielhoefer1982,Kielhoefer1979,Kielhoefer1979a} under the assumption of  analyticity. The main results are summarized in his monograph \cite[I.17]{Kielhoefer2012}. Assuming the analyticity of $F$ and $\rea \mu^{\prime}(0)=\cdots=\rea\mu^{(m-1)}(0)=0$ but $\rea\mu^{(m)}(0)\neq 0$ for some $m>1$,
Kielh\"{o}fer proved that if $m$ is odd, then problem \eqref{a1} admits at least one and at most $m$ nontrivial bifurcaton curves of periodic solutions; if $m$ is even, Hopf bifurcation may or may not occur (cf. \cite[Theorem I.17.3]{Kielhoefer2012}). The main tool is the method of Newton's polygon.
For finite dimensional problems, degenerate Hopf bifurcations under (F4$^{\prime}$) also had been extensively studied (cf. \cite{Chafee1968}, \cite{Schmidt1978} and \cite[VIII.5]{Golubitsky1985}).
To the authors' best knowledge, no general results of problem \eqref{a1} in infinite dimensions  under (F4$^{\prime}$) are established without the assumption of the analyticity of $F$.

Our main purpose of this paper is to establish an abstract Hopf bifurcation theorem in infinite dimensions under the degeneracy (F4$^{\prime}$) but no assuming the analyticity of $F$.
We show that in our situation a transcritical Hopf bifurcation still can happen at $\lambda_0$ although the stability property of the trivial solution line does not change  near $\lambda_0$.
We also obtain the stability of bifurcating periodic solutions. 
The main tools are a Lyapunov–-Schmidt reduction and an improved Morse lemma, and the latter comes from Liu et al. \cite{Liu2007} (also see Chang \cite[Theorem I.5.1]{Chang1993} or Kuiper \cite{Kuiper1972}).
The classic Morse lemma was firstly applied to bifurcation problems in Nirenberg \cite[Chapter 3]{Nirenberg2001}. Our research is inspired by the works of Liu et al. \cite{Liu2007,Liu2013a}, where the improved Morse lemma was applied to stationary bifurcation problems in infinite dimensions when the \emph{transversality} condition fails,
and two abstract bifurcation theorems were established for the saddle-node bifurcation and bifurcation from a degenerate simple eigenvalue, respectively.
As those in \cite{Crandall1977,Liu2007,Liu2013a},  our abstract theorems also directly deal with the original
 infinite dimensional problems, the conditions depend on some Fr\'{e}chet derivatives of $F$, do not involve reduced finite-dimensional problems and hence is convenient to use.
As shown in some applications,
 our new theorems can be viewed as a supplement of the classic Hopf bifurcation theorem in infinite dimensions and  usually provide ``hidden'' new  branches of bifurcating periodic solutions, by rechecking the range of parameters under the \emph{degeneracy} condition.

We organize the paper as follows.  We state our main results in Section 2. In Section 3, we prove the degenerate Hopf bifurcaiton theorem --- Theorem \ref{thm2}.  Section 4 contains
the proof of the stability theorem --- Theorem \ref{thm:stable2}.
Finally in Section 5, we give some applications of the main results to a diffusive predator--prey system.

\section{Main Results}
\label{sec:2}

\noindent\textbf{Notations } Let $\varphi_0$ be the eigenvector of $A_0$ given in (F3).
Denote by $\langle\cdot,\varphi^{\ast}_0\rangle$ a bounded linear functional such that $\langle\varphi_0,\varphi^{\ast}_0\rangle=1$ and $\langle z,\varphi^{\ast}_0\rangle=0$ for every $z\in R(i\kappa_0I-A_0)$, guaranteed by the Hahn-Banach Theorem. Set
\begin{equation}\label{eq:phi}
\psi_0(t)=\varphi_0 e^{it}\quad\text{and} \quad\psi^{\ast}_0(t)=\varphi^{\ast}_0e^{-it}.
\end{equation}
Denote by $D_xF[\cdot]$, $D^2_{xx}F[\cdot,\cdot]$, $D^3_{xxx}F[\cdot,\cdot,\cdot]$, $D^2_{x\lambda}F[\cdot]$ and $D^2_{x\kappa}F[\cdot]$ the Frech\'{e}t derivatives and partial derivatives of $F$ with respect to $x,\lambda$ and $\kappa$ (the last two variables are one-dimensional parameters). They are linear or multilinear operators acting on functions in $[~]$ or $1$.
Set
\begin{align} \label{eq:HH}
\begin{split}
H_{11}&=\rea \big\langle-D_{xxx}^3F(0,\lambda_0)[\varphi_0,\varphi_0,\overline{\varphi}_0]
+2D_{xx}^2F(0,\lambda_0)\big[\varphi_0,A_0^{-1}D_{xx}^2F(0,\lambda_0)[\varphi_0,\overline{\varphi}_0]\big]\\
&\quad -D_{xx}^2F(0,\lambda_0)\big[\overline{\varphi}_0,(2i\kappa_0 I-A_0)^{-1}D_{xx}^2F(0,\lambda_0)[\varphi_0,\varphi_0]\big], \varphi^{\ast}_0\big\rangle,\\
H_{22}&=\rea \big\langle-D_{x\lambda\lambda}^3F(0,\lambda_0)[\varphi_0]-2D_{x\lambda}^2F(0,\lambda_0)\big[(i\kappa_0I-A_0)^{-1}\big(D_{x\lambda}^2F(0,\lambda_0)[\varphi_0]\\
&\quad-\langle D_{x\lambda}^2F(0,\lambda_0)[\varphi_0],\varphi^{\ast}_0\rangle\varphi_0\big)\big],\varphi^{\ast}_0\big\rangle.
\end{split}
\end{align}
Under conditions (F1) and (F3), $H_{11}$ and $H_{22}$ are well defined. In fact, $(i\kappa_0I-A_0)^{-1}$ maps $R(i\kappa_0I-A_0)$ to $R(i\kappa_0I-A_0)\cap X$ because
the algebraic simplicity of $i\kappa_0$ implies the decompositions $Z=R(i\kappa_0I-A_0)\oplus\spa\{\varphi_{0}\}$ and $X=(R(i\kappa_0I-A_0)\cap X)\oplus\spa\{\varphi_{0}\}$.

The following results are our main theorems on the degenerate Hopf bifurcation. 
\begin{thm}[See Fig.\ref{fig:surface}]\label{thm2}
Consider problem \eqref{a1} under conditions (F1$^{\prime}$)(F2)(F3)(F4$^{\prime}$)(F5) (F6).
Set the diagonal matrix $H_0=\Diag(H_{11},H_{22})$.
Then the following assertions hold:
\begin{enumerate}[(1)]
  \item If $H_0$ is definite, i.e., $\det H_0=H_{11}H_{22}>0$, then the set of periodic solutions of \eqref{a1} near $(x,\lambda)=(0,\lambda_0)$ is the trivial solution line $\{(0, \lambda)\} $.
  \item If $H_0$ is indefinite, i.e., $\det H_0=H_{11}H_{22}<0$, then apart from the trivial solution line $\{(0, \lambda)\} $,
       there exists a continuously differentiable curve $\{(x(r), \lambda(r))\}$ of $2 \pi / \kappa(r)$-periodic solutions of \eqref{a1} through the bifurcation point $(x(0), \lambda(0))=(0, \lambda_{0})$ with $\kappa(0)=\kappa_{0}$ in $\big(C_{2 \pi / \kappa(r)}^{1+\alpha}(\mathbb{R}, Z) \cap C_{2 \pi / \kappa(r)}^{\alpha}(\mathbb{R}, X)\big) \times \mathbb{R}$.
Every other periodic solution of \eqref{a1} in a neighborhood of $\left(0, \lambda_{0}\right)$ is obtained from $(x(r), \lambda(r))$ by a phase shift $S_{\theta} x(r)$.
Moreover, the tangent vector of $\{(x(r), \lambda(r))\}$ with $2 \pi / \kappa(r)$-period at $(0, \lambda_{0})$ is given by
\begin{equation}\label{eq:tangent1}
\big(\dot{x}(0),\dot{\lambda}(0)\big)=\big(2\rea (\varphi_0 e^{i\kappa_0 t}),\eta\big)\quad \hbox{with } \dot{\kappa}(0)=\im \langle D_{x\lambda}^2F(0,\lambda_0)[\varphi_0],\varphi^{\ast}_0\rangle\eta,
\end{equation}
where $\eta=\sqrt{-{H_{11}}/{H_{22}}}$. 
\end{enumerate}
\begin{figure} 
\centering
\includegraphics[totalheight=1.5in]{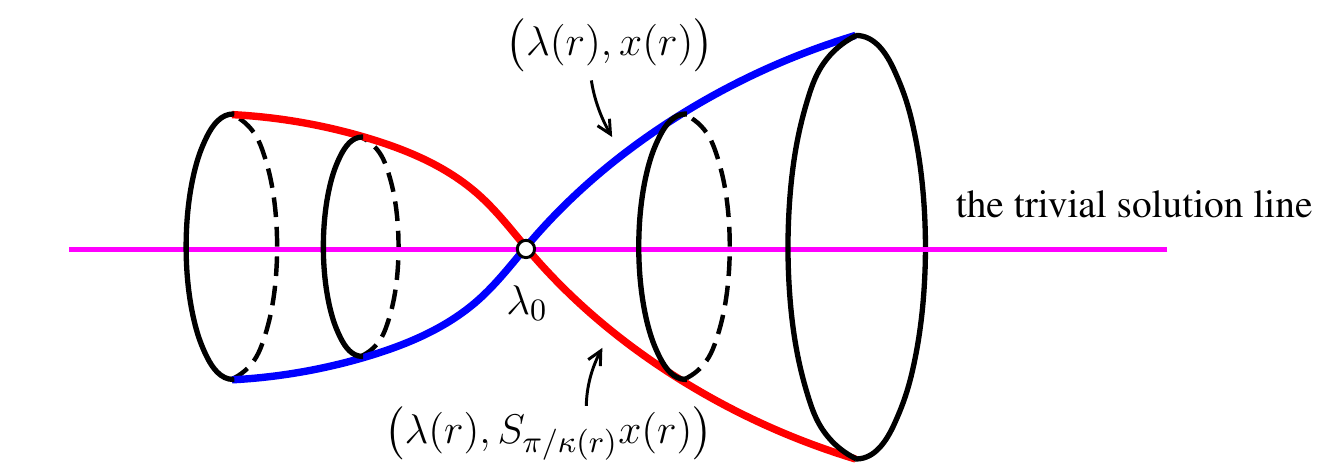}
\caption{The local surface composed of bifurcating periodic solutions in Theorem \ref{thm2}(2). Transcritical Bifurcation.
Every other periodic solution of \eqref{a1} in a neighborhood of $\left(0, \lambda_{0}\right)$ is derived from $(x(r), \lambda(r))$ by a phase shift $S_{\theta} x(r)$, 	e.g., the branch (red) is obtained by phase shift $S_{\pi/\kappa(r)}$ from the branch $(\{x(r), \lambda(r))\}$ (blue) and vice verse.}\label{fig:surface}
\end{figure}
\end{thm}

The next result gives the stability properties of the trivial solution and the bifurcating periodic solutions in the degenerate Hopf bifurcation.
\begin{thm}[See Fig.\ref{fig:1mems}]\label{thm:stable2}
Under the conditions of Theorem \ref{thm2}, further assume that $H_0$ is indefinite and (F7) holds. Let $\{(x(r), \lambda(r)) \mid r \in(-\delta, \delta)\}$ be the curves of $2 \pi / \kappa(r)$-periodic solutions of \eqref{a1} according to Theorem \ref{thm2}. If $H_{22}>0$, then the trivial solution $\{(0,\lambda)\}$ is stable near $(0,\lambda_0)$, and the bifurcating periodic solutions $\{(x(r),\lambda(r))\}$ with period $2\pi/\kappa(r)$, are unstable near $(0,\lambda_0)$.
If $H_{22}<0$, then the stability property of the trivial solution is reversed and that of the bifurcating periodic solutions is also reversed.
\end{thm}
\begin{figure} 
\centering
\includegraphics[totalheight=1.6in]{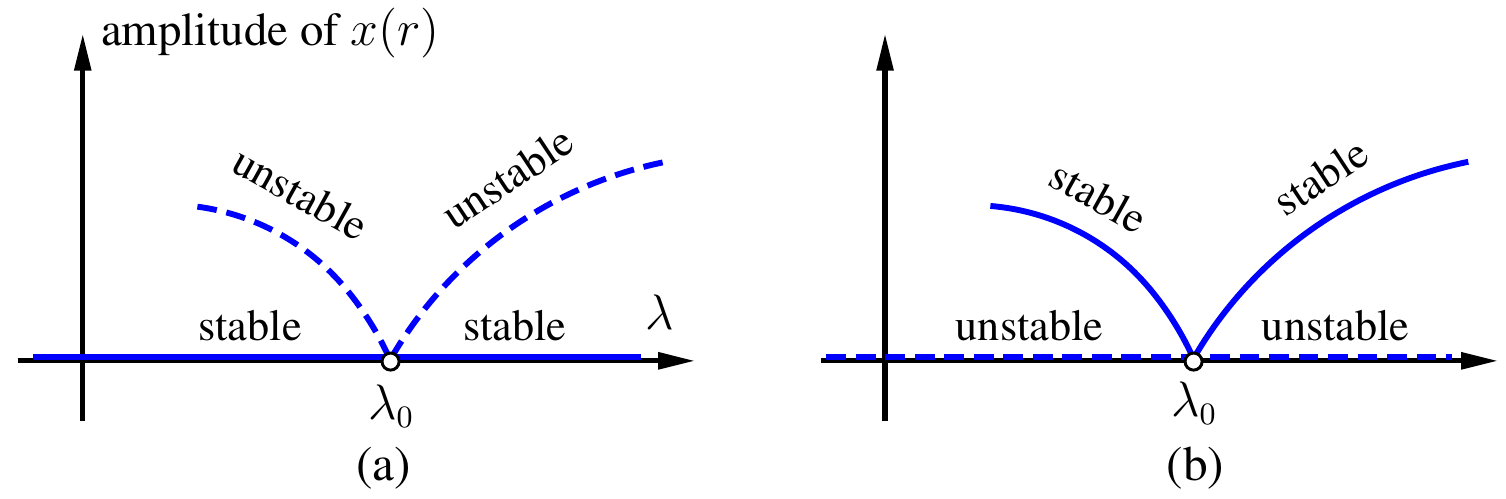}
\caption{The bifurcating periodic solutions with small amplitude in Theorem \ref{thm2}(2) are unstable if (a) $H_{22}>0$ and stable if (b) $H_{22}<0$. Stability properties of the trivial solutions are reversed. We depict only two half-branches and all other periodic solutions are obtained from these two half-branches via phase shifts.}\label{fig:1mems}
\end{figure}

Clearly, the \emph{Principle of Exchange of Stability} still holds
 for this kind of degenerate Hopf bifurcation. Since adjacent solutions have different stability properties,
it brings the convenience of deriving the stability properties of bifurcation solutions directly from those of trivial solutions, and vice versa.

We also obtain the following useful result, which makes the meaning of $H_{22}$ more clear.
\begin{cor}\label{cor:h22}
Under the conditions of Theorem \ref{thm2},	$H_{22}=-\rea\mu^{\prime\prime}(\lambda_0)$.
\end{cor}

Now, we know that conditions (F4$^{\prime}$) and $H_{22}\neq 0$ actually mean that the eigenvalue $\mu(\lambda)$ goes through the imaginary axis at $\lambda_0$  with ``vanishing speed'' but ``nonvanishing acceleration''. 
This relation also enables to simplify the calculation of $H_{22}$ in many applications  by using $-\rea\mu^{\prime\prime}(\lambda_0)$ instead of \eqref{eq:HH}$_2$. 

\begin{rem}
	Theorems \ref{thm2} and \ref{thm:stable2} are consistent with the known results proved in \cite{Kielhoefer2012} under the analyticity of $F$.
	Our results recover those of the case $m=2$ in \cite[(I.17.28) and below]{Kielhoefer2012} and \cite[Theorem I.17.4]{Kielhoefer2012}.
	The advantage is that no assumptions of the analyticity are required.
	Although we here do not study the higher-order degenerate cases (i.e., $m>2$). However, our theorems actually deal with the most important and commonly used case outside the previous Theorems \ref{thm1} and \ref{thm:stable1}.
\end{rem}

Let us illustrate the above theorems by a simple example of ODE system.
\begin{example}\label{eg:simu}
Consider periodic solutions of the following problem
\begin{equation}\label{eq:simple}
\frac{\mathrm{d}}{\mathrm{d}t}
\begin{pmatrix}
x_1(t)\\x_2(t)
\end{pmatrix}=
\begin{pmatrix}
\beta(\lambda)&1\\
-1&\beta(\lambda)
\end{pmatrix}
\begin{pmatrix}
x_1(t)\\x_2(t)
\end{pmatrix}+\begin{pmatrix}
x_1^3(t)\\x_2^3(t)
\end{pmatrix},
\end{equation}
where $x_1(\cdot),x_2(\cdot)\in C^1(\mathbb{R},\mathbb{R})$, and $\beta(\cdot)\in C^2(\mathbb{R},\mathbb{R})$ satisfying $\beta(\lambda_0)=0$.
\end{example}
Set $x=(x_1,x_2)^T$. Clearly, $(x,\lambda)=(0,\lambda)$ is the trivial solution of \eqref{eq:simple} for all $\lambda \in \mathbb{R}$.
Denote
\begin{equation*}
F(x,\lambda)=
\begin{pmatrix}
\beta(\lambda)&1\\
-1&\beta(\lambda)
\end{pmatrix}
\begin{pmatrix}
x_1\\x_2
\end{pmatrix}+
\begin{pmatrix}
x_1^3\\x_2^3
\end{pmatrix}.
\end{equation*}
Take $X=Z=\mathbb{R}^2$.
Then $A=D_xF(0,\lambda)=\begin{small}
\begin{pmatrix}
\beta(\lambda)&1\\
-1&\beta(\lambda)
\end{pmatrix}:
\end{small}
\mathbb{R}^2\rightarrow\mathbb{R}^2$.
The eigenvalues of $A$ are $\beta(\lambda)\pm i$. So $\pm i$ are simple eigenvalues of $A_0=D_xF(0,\lambda_0)=\begin{small}
\begin{pmatrix}
  & 1\\
 -1 &
\end{pmatrix}
\end{small}$ with eigenvectors
$\varphi_0,\overline{\varphi}_0=\frac{1}{\sqrt{2}}(
1,\pm i)^T$,  respectively. Moreover, the unique perturbed eigenvalues $\mu(\lambda)$ of $D_xF(0,\lambda)$ near $\pm i$ are also $\mu(\lambda)=\beta(\lambda)\pm i$, respectively. Clearly, $\varphi^{\ast}_0=\varphi_0$ and $\overline{\varphi}^{\ast}_0=\overline{\varphi}_0$, associated with the inner product in $\mathbb{C}^2$ as the bounded linear functional.
Here, $\mathbb{C}^2$ is the complexified space of $\mathbb{R}^2$.
The associated semigroup is $x(t)=e^{A_0t}$, which is compact for $t>0$.
Since
\begin{equation*}
D_{xx}^2F(0,\lambda_0)\Big[\begin{pmatrix}
a_1\\
a_2
\end{pmatrix}, \begin{pmatrix}
b_1\\
b_2
\end{pmatrix}\Big]=
\begin{pmatrix}
0\\0
\end{pmatrix},
\; D_{xxx}^3F(0,\lambda_0)\Big[\begin{pmatrix}
a_1\\
a_2
\end{pmatrix}, \begin{pmatrix}
b_1\\
b_2
\end{pmatrix}, \begin{pmatrix}
c_1\\
c_2
\end{pmatrix}\Big]=
\begin{pmatrix}
6 a_1  b_1  c_1\\
6 a_2  b_2  c_2
\end{pmatrix},
\end{equation*}
it follows from the expression \eqref{eq:HH} that
\begin{equation*}
\begin{split}
H_{11}=\rea\langle -D_{xxx}^3F(0,\lambda_0)[\varphi_0,\varphi_0,\overline{\varphi}_0], \varphi^{\ast}_0\rangle
=-\frac{3}{2}\rea\left\langle
\begin{pmatrix}
1\cdot 1\cdot 1\\
i\cdot i\cdot (-i)
\end{pmatrix},
\begin{pmatrix}
1\\
i
\end{pmatrix}
\right\rangle
=-3<0.
\end{split}
\end{equation*}

Now we apply Theorem \ref{thm2} and \ref{thm:stable2} to analyze several specific cases of $\beta(\lambda)$ in \eqref{eq:simple}.
\begin{enumerate}[(1)]
  \item Let $\beta(\lambda)=\lambda-\lambda_0$. Since $\rea\mu^{\prime}(\lambda_0)=\beta^{\prime}(\lambda_0)=1\ne 0$, it follows that condition (F4) holds.
      By Theorem \ref{thm1},
      a nondegenerate  Hopf bifurcation happens at $(0,\lambda_0)$. See the left of Fig.\ref{fig:simulation} for numerical simulations of the bifurcation curve and some limit cycles (periodic solutions) near $\lambda_0=0$.
  \item Let $\beta(\lambda)=(\lambda-\lambda_0)^2$. Since $\rea\mu^{\prime}(\lambda_0)=\beta^{\prime}(\lambda_0)=0$ and $H_{22}=-\rea\mu^{\prime\prime}(\lambda_0)=-\beta^{\prime\prime}(\lambda_0)=-2<0$,  it follows that $\det H_0=6>0$. Since $H_0$ is definite, Theorem \ref{thm2}(1) implies that for periodic solutions near $(0,\lambda_0)$, problem \eqref{eq:simple} admits only the trivial solution line.
  \item Let $\beta(\lambda)=-(\lambda-\lambda_0)^2$. Since $\rea\mu^{\prime}(\lambda_0)=\beta^{\prime}(\lambda_0)=0$ and $H_{22}=-\rea\mu^{\prime\prime}(\lambda_0)=-\beta^{\prime\prime}(\lambda_0)=2>0$.
      it follows that $\det H_0=-6<0$. Since (F4$^{\prime}$) holds and $H_0$ is indefinite,  it follows from Theorem \ref{thm2}(2) that a degenerate Hopf bifurcation happens, i.e.,
      apart from the trivial solution line, there exists a continuously differentiable curve $\{(x(r), \lambda(r))\}$ of $2 \pi / \kappa(r)$-period through $(x(0), \lambda(0))=\left(0, \lambda_{0}\right)$ with $\kappa(0)=1$ in $C_{2 \pi / \kappa(r)}^{1}(\mathbb{R}, \mathbb{R}^2) \times \mathbb{R}$.
      Every other periodic solution of \eqref{a1} near $\left(0, \lambda_{0}\right)$ is obtained from $\{(x(r), \lambda(r))\}$ by a phase shift $S_{\theta} x(r)$.
       Moreover, by Theorem \ref{thm:stable2}, the bifurcating periodic solution $\{(x(r), \lambda(r))\}$  is unstable and the trivial solution is stable; see Figure \ref{fig:1mems}(a). Also, see the right of Figure \ref{fig:simulation} for numerical simulations of the bifurcation curve and some of periodic solutions (limit cycles) near $\lambda_0=0$. \hfill $\Box$
\end{enumerate}
\begin{figure} 
\centering
\includegraphics[totalheight=2.2in]{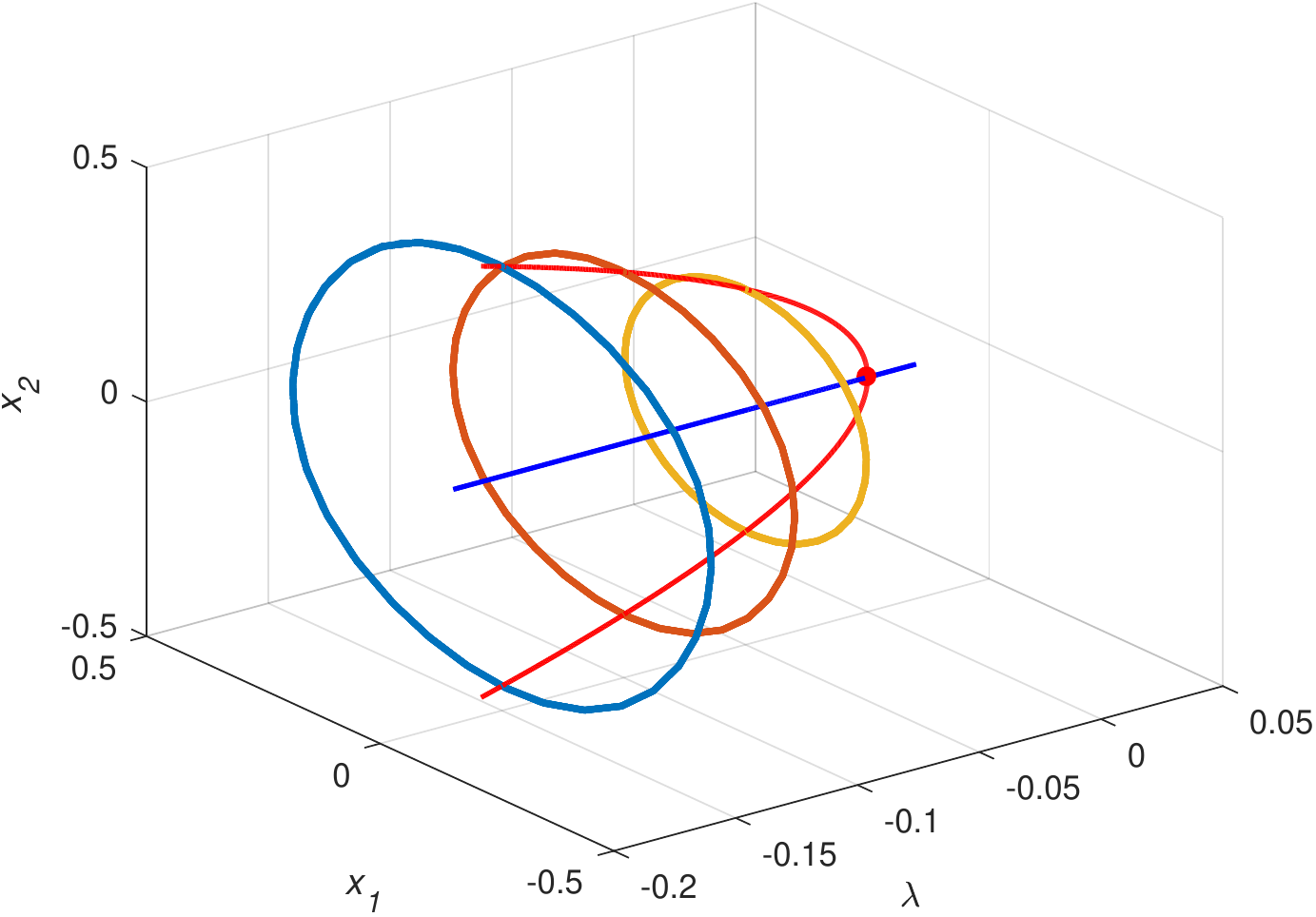}
\includegraphics[totalheight=2.2in]{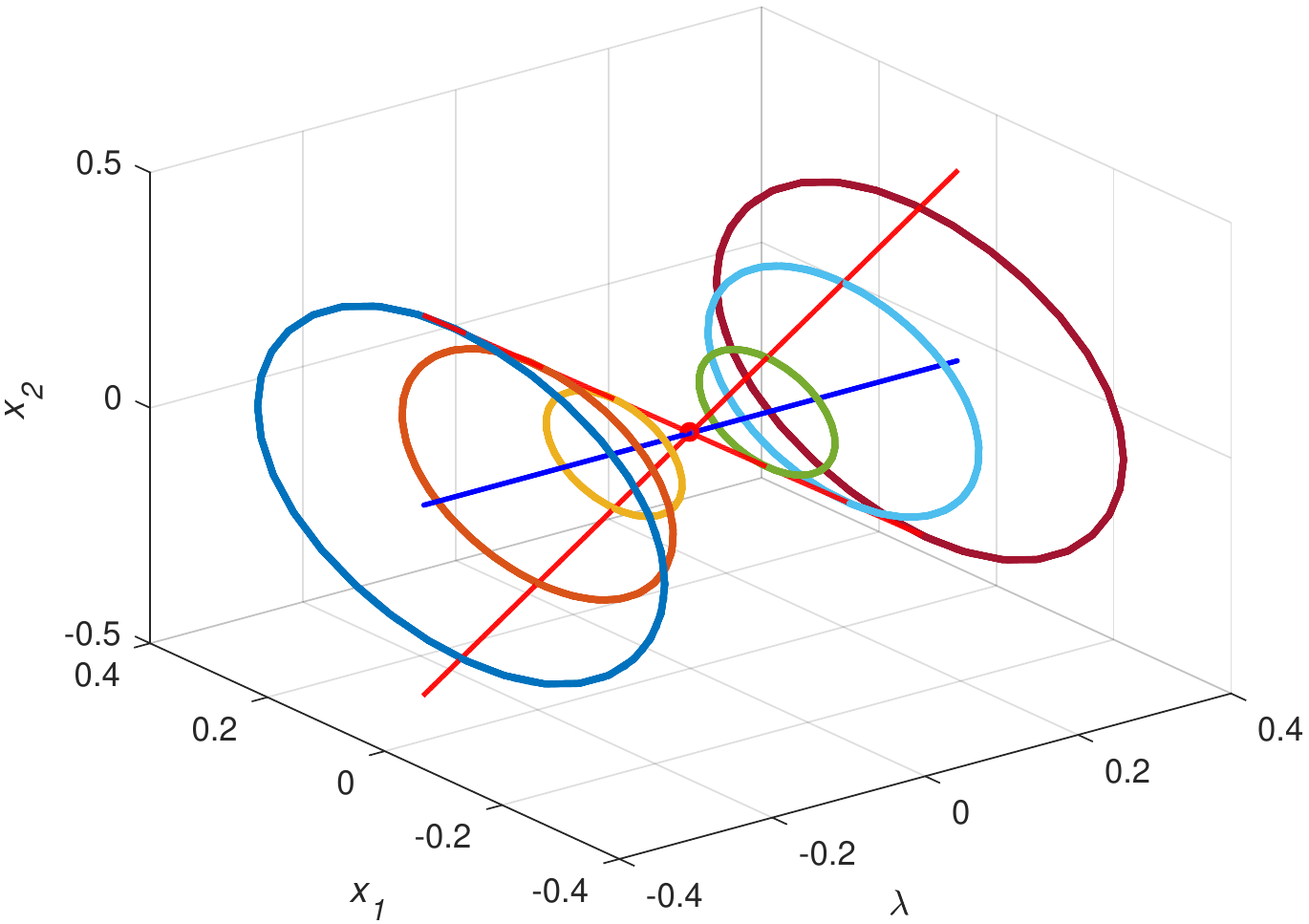}
\caption{Numerical simulations for nondegenerate and degenerate Hopf bifurcations in Example \ref{eg:simu}(1) and (3), respectively. Here, $\lambda_0=0$, and for each given small $r>0$ and $\lambda$ nearby $\lambda_0$, there exists a periodic solution (limit cycle) with amplitude $r$.
Left: nondegenerate Hopf bifurcation, periodic solutions for $\lambda=-0.15,-0.1,-0.05$.
Right: degenerate Hopf bifurcation, periodic solutions for $\lambda=-0.3,-0.2,-0.1,0.1,0.2,0.3$.} \label{fig:simulation}
\end{figure}

We are mainly conerned with bifurcation problems in infinite dimensions. Applications of Theorems \ref{thm2} and \ref{thm:stable2} to PDEs will be given in Section 5.

\section{Proof of Theorem \ref{thm2}}\label{sec:3}
In this section, we prove the degenerate Hopf bifurcation theorem --- Theorem \ref{thm2}.
Before proceeding, we recall some known results on Hopf bifurcation of problem \eqref{a1} from Kielh\"{o}fer \cite[I.8]{Kielhoefer2012} for the reader's convenience. Our proof will be based on those results, and use the same notations (except for $\varphi_0^*$ instead of $\varphi_0^\prime$ therein to avoid confusion with the derivative).

\subsection{Preliminaries}
Since $x=x(t)$ is a solution of $\kappa \frac{\mathrm{d}x}{ \mathrm{d} t}=F(x,\lambda)$ with $2\pi$-period if and only if $\tilde{x}=x(\kappa t)$ is a solution of \eqref{a1} with $2\pi/\kappa$-period,
it is equivalent to studying the $2\pi$-period solution of the equation
\begin{equation}\label{b1}
G(x,\kappa,\lambda)\equiv \kappa \frac{\mathrm{d}x}{ \mathrm{d} t}-F(x,\lambda)=0\quad \hbox{near } (0, \kappa_{0}, \lambda_{0}),
\end{equation}
where $G:\tilde{U}\times\tilde{V}\to W$, $
0\in \tilde{U}\subset E\cap Y\text{ and }(\kappa_0,\lambda_0)\in\tilde{V}\subset\mathbb{R}^2$ are open neighborhoods.
Here, Banach spaces $E\equiv C_{2\pi}^\alpha(\mathbb{R},X)$, $W\equiv C_{2\pi}^\alpha(\mathbb{R},Z)$, and $Y\equiv C_{2\pi}^{1+\alpha}(\mathbb{R},Z)$.
It is known that
$E \cap Y$ is a Banach space with norm $\|x\|_{X,\alpha}+\|\frac{\mathrm{d}x}{\mathrm{d} t}\|_{Z,\alpha}$ and is continuously embedded into $W$.
Moreover, $F \in C^{3}(U \times V, Z)$ by (F1) implies that $G \in C^{2}(\tilde{U} \times \tilde{V}, W)$.

Under (F3), (F5) and (F6), it follows from \cite[Proposition I.8.1]{Kielhoefer2012} that the linear operator
$J_{0} \equiv D_{x}G(0, \kappa_{0}, \lambda_{0})=\kappa_{0} \frac{\mathrm{d}}{ \mathrm{d} t}-A_{0}: E \cap Y \rightarrow W$
is continuous and is a Fredholm operator of index zero, with $\dim N(J_{0})=2$. Moreover,
\begin{equation}\label{eq:span}
 W=R(J_{0})\oplus N(J_{0})\quad  \hbox{ and } \quad N(J_{0})=\{c \varphi_{0} e^{i t}+\overline{c}\,\overline{\varphi}_{0} e^{-i t} \mid c \in \mathbb{C}\}.
\end{equation}
The Lyapunov-Schmidt reduction is applicable to $G$ at $(0, \kappa_{0}, \lambda_{0})$ under the decomposition \eqref{eq:span}.
Then \eqref{b1} is reduced to the following form
\begin{equation}\label{eq:redu}
 Q G(v+\psi(v, \kappa, \lambda), \kappa, \lambda)=0,\quad v=Px,
\end{equation}
where the projections
\begin{align}\label{Q}
&Q: W \rightarrow N(J_{0})(=R(Q))\quad \hbox{ along } R(J_{0})(=N(Q)),\\ \nonumber
&P=Q|_{E \cap Y}: E \cap Y \rightarrow N(J_{0})\quad \hbox{ along } R(J_{0})\cap(E \cap Y),
\end{align}
and $\psi(v,\kappa,\lambda)$ is determined locally by the Implicit Function Theorem, solving $(I-Q) G(Px+(I-P)x, \kappa, \lambda)=0$ for $(I-P)x$ on $R(J_{0})$.
In addition, 
\begin{align}\label{eq:psi0}
\begin{split}
&{D}_v\psi(0,\kappa_0,\lambda_0)=0 \text{ and}\\
&\psi(0,\kappa,\lambda)=D_{\kappa}\psi(0,\kappa,\lambda)=D_{\lambda}\psi(0,\kappa,\lambda)=D_{\kappa\lambda}\psi (0,\kappa,\lambda)=0\quad
\text{ near }(\kappa_0,\lambda_0).
\end{split}
\end{align}
Moreover, for each real $z(t)\in W$, the projection $Q$ has an expression
\begin{equation}\label{eq:Q}
(Qz)(t)=\frac{1}{2\pi}\Big(\int_0^{2\pi}\langle z(t),\psi^{\ast}_0(t)\rangle  \mathrm{d} t \Big) \psi_0(t)+\frac{1}{2\pi}\Big(\int_0^{2\pi}\langle z(t),\overline{\psi}{}^{\ast}_{0}(t)\rangle  \mathrm{d} t \Big) \overline{\psi}_0(t),
\end{equation}
where $\psi_0$ and $\psi^{\ast}_0$ are given as in \eqref{eq:phi}. A real function $Px\in N(J_0)$ has the following form
$$(Px)(t) =c\varphi_{0} e^{i t}+\overline{c}\,\overline{\varphi}_{0} e^{-i t}, \quad c \in \mathbb{C}.$$
Then \eqref{eq:Q} implies that  in view of \eqref{eq:span},
for a real function $Px\in N(J_{0})$, two-dimensional system \eqref{eq:redu} is equivalent to the one-dimensional complex equation
\begin{equation}\label{b4}
\hat{\Phi}(c,\kappa,\lambda)\equiv\frac{1}{2\pi}\int_0^{2\pi}\langle G(Px+\psi(Px,\kappa,\lambda),\kappa,\lambda),\psi^{\ast}_0\rangle  \mathrm{d} t=0.
\end{equation}
In order to solve the equation
\begin{align*} 
\hat{\Phi}(c,\kappa,\lambda)=0, \quad\text{ where }
  \hat{\Phi}:\tilde{U}_2\times\tilde{V}_2\to \mathbb{C},\;
0\in \tilde{U}_2\subset \mathbb{C} \text{ and }(\kappa_0,\lambda_0)\in\tilde{V}\subset\mathbb{R}^2,
\end{align*}
by $S^{1}$ equivariance of $\hat{\Phi}$ (cf. \cite[(I.8.36)]{Kielhoefer2012}), i.e.,
\begin{equation*} 
	\hat{\Phi}(e^{i\theta}c,\kappa,\lambda)=e^{i\theta}\hat{\Phi}(c,\kappa,\lambda), \quad \theta\in [0,2\pi),
\end{equation*}
it suffices to solve
\begin{equation}\label{eq:bifeq}
 \hat{\Phi}(r, \kappa, \lambda)=0, \quad r \in(-\delta, \delta) \subset \mathbb{R}, \; (\kappa, \lambda) \in \tilde{V}_{2} \subset \mathbb{R}^{2},
\end{equation}
 which is referred to as the \emph{bifurcation equation for Hopf bifurcation}. Now, we can write
 $Px$ in the form 
 $$(Px)(t)=r(\varphi_0 e^{it}+\overline{\varphi}_0e^{-it}).$$ 

Equation \eqref{eq:bifeq} admits a trivial solution, i.e., $\hat{\Phi}(0, \kappa, \lambda)=0$ for all $(\kappa, \lambda) \in \tilde{V}_{2}$.
In order to remove the trivial solution, it is convenient to set 
$$\tilde{\Phi}(r, \kappa, \lambda)=\hat{\Phi}(r, \kappa, \lambda)/r\quad \text{ for } r\neq 0$$ and investigate the equation
\begin{equation} \label{b0}
\tilde{\Phi}(r, \kappa, \lambda)\equiv\int_0^1{D}_r\hat{\Phi}(\tau r,\kappa,\lambda)\mathrm{d} \tau=0,\quad r \in(-\delta, \delta).
\end{equation}
Direct computation yields
\begin{align}\label{eq:matrix1}
\begin{split}
   &  \tilde{\Phi}\left(0, \kappa_{0}, \lambda_{0}\right)=0,\\
   & {D}_{(\kappa, \lambda)} \tilde{\Phi}\left(0, \kappa_{0}, \lambda_{0}\right)=\left(\begin{array}{ll}
0 & -\rea \left\langle {D}_{x \lambda}^{2} F\left(0, \lambda_{0}\right) \varphi_{0}, \varphi^{\ast}_0\right\rangle \\
1 & -\im \left\langle {D}_{x \lambda}^{2} F\left(0, \lambda_{0}\right) \varphi_{0}, \varphi^{\ast}_0\right\rangle
\end{array}\right).
\end{split}
\end{align}
This implies that if the following\emph{ transversality} condition holds:
\begin{enumerate}[(F4$^{\prime\prime}$)]
\item $\rea\left\langle {D}_{x \lambda}^{2} F\left(0, \lambda_{0}\right) \varphi_{0}, \varphi^{\ast}_0\right\rangle \neq 0,$
\end{enumerate}
then \eqref{b0} can be solved by the Implicit Function Theorem. 
Notice that (F4$^{\prime\prime}$) is equivalent to the \emph{nondegeneracy} condition (F4). Indeed, differentiating \eqref{eq:perteigen} with respect to $\lambda$ at $\lambda_0$ yields (\text{cf.} \cite[\text{(I.8.44)}$_4$]{Kielhoefer2012})
\begin{equation}\label{eq:equivalent}
\mu^{\prime}(\lambda_0)=\langle {D}^2_{x\lambda}F(0,\lambda_0)\varphi_0,\varphi^{\ast}_0\rangle.
\end{equation}
Thus, Theorem \ref{thm1} can be proved, as shown in \cite[I.8]{Kielhoefer2012}

However, the present paper deals with problem \eqref{a1} under condition (F4$^{\prime}$). That is, the nondegeneracy (F4) is violated. 
The Implicit Function Theorem now is no longer directly applicable to \eqref{eq:matrix1}. Therefore, we cannot obtain Theorem \ref{thm2} in the same way as Theorem \ref{thm1}.
We will use a Morse lemma to prove the theorem of the degenerate Hopf bifurcation.

\subsection{Lemmas}

To prove the theorem, we need two lemmas. Following \cite{Kielhoefer2012}, we introduce several notations of the real vectors (cf. \eqref{eq:phi})
\begin{equation}\label{eq:vv}
\begin{split}
&\hat{v}_{1}\equiv i\left(\psi_{0}-\overline{\psi}_{0}\right)=-2\im \psi_{0}, \quad \hat{v}_{2}\equiv \psi_{0}+\overline{\psi}_{0}=2\rea \psi_{0}, \\
&\hat{v}^{\ast}_{1}\equiv -\frac{i}{2}\left(\psi^{\ast}_{0}
-\overline{\psi}{}^{\ast}_{0}\right)=\im \psi^{\ast}_{0}, \quad \hat{v}^{\ast}_{2}\equiv \frac{1}{2}\left(\psi^{\ast}_{0}
+\overline{\psi}{}^{\ast}_{0}\right)=\rea \psi^{\ast}_{0}.
\end{split}
\end{equation}

The following lemma from Liu et al. \cite[Lemma 2.5]{Liu2007}, where the result was proved by the classic Morse Lemma and the invariant manifold theory.

\begin{lem}[\cite{Liu2007}]\label{lem:morse}
	Suppose that $(x_0,y_0)\in\mathbb{R}^2$ and $U$ is a neighborhood of $(x_0,y_0)$. Assume that $f: U\rightarrow\mathbb{R}$ is a $C^p$ function for $p\geqslant 2$, $f(x_0,y_0)=0$, ${D}  f(x_0,y_0)=0$, and the Hessian matrix $H={D} ^2f(x_0,y_0)$ is nondegenerate. Then
	\noindent
	\begin{enumerate}[(1)]
		\item If $H$ is definite (i.e., $\det H >0$), then $(x_0,y_0)$ is the unique zero point of $f(x,y)=0$ near $(x_0,y_0)$.
		\item If $H$ is indefinite (i.e., $\det H <0$), then there exist two $C^{p-1}$ curves $(x_j(t),y_j(t))$, $j=1,2$, $t\in(-\delta,\delta)$, such that the solution set of $f(x,y)=0$ consists of exactly the two curves near $(x_0,y_0)$, $(x_j(0),y_j(0))=(x_0,y_0)$. Moreover, t can be rescaled and indices can be rearranged so that $(x_1^{\prime}(0),y_1^{\prime}(0))$ and $(x_2^{\prime}(0),y_2^{\prime}(0))$ are two linear independent solutions of
		\be\label{b23}
		f_{xx}(x_0,y_0)\eta^2+2f_{xy}(x_0,y_0)\eta\tau+f_{yy}(x_0,y_0)\tau^2=0.
		\ee
	\end{enumerate}
\end{lem}

\begin{rem}
	As pointed out in \cite{Liu2007}, $C^{p-1}$ in Lemma \ref{lem:morse}(2) is the optimal regularity of the curves.
	In contrast, the known result of two $C^{p-2}$ curves was known earlier in Nirenberg \cite[Corollary  3.1.2]{Nirenberg2001}.
\end{rem}

The next result gives the 2nd and 3rd derivatives of $\hat{\Phi}$ at $(0,\kappa_0,\lambda_0)$ as well as the relations between $H_{ii}$ and $\hat{\Phi}$.
\begin{lem}\label{lem:3.1}
Let $\hat{\Phi}(r,\kappa,\lambda)$ be defined as in \eqref{b4} and $H_{11}, H_{22}$ as in \eqref{eq:HH}. Then
\begin{align*}
{D}^2_{rr}\hat{\Phi}(0,\kappa_0,\lambda_0)&=0,\quad {D}^2_{r\kappa}\hat{\Phi}(0,\kappa_0,\lambda_0)
=i,\quad {D}^2_{r\lambda}\hat{\Phi}(0,\kappa_0,\lambda_0)=-\langle D_{x\lambda}^2F(0,\lambda_0)[\varphi_0],\varphi^{\ast}_0\rangle,\\
{D}^3_{rr\kappa}\hat{\Phi}(0,\kappa_0,\lambda_0)&={D}^3_{rr\lambda}\hat{\Phi}(0,\kappa_0,\lambda_0)
={D}^3_{r\kappa\kappa}\hat{\Phi}(0,\kappa_0,\lambda_0)={D}^3_{r\kappa\lambda}\hat{\Phi}(0,\kappa_0,\lambda_0)=0,\\
{D}^3_{rrr}\hat{\Phi}(0,\kappa_0,\lambda_0)
&=3\big\langle-D_{xxx}^3F(0,\lambda_0)[\varphi_0,\varphi_0,\overline{\varphi}_0]
+2D_{xx}^2F(0,\lambda_0)\big[\varphi_0,A_0^{-1}D_{xx}^2F(0,\lambda_0)[\varphi_0,\overline{\varphi}_0]\big]\\
&\quad -D_{xx}^2F(0,\lambda_0)\big[\overline{\varphi}_0,(2i\kappa_0 I-A_0)^{-1}D_{xx}^2F(0,\lambda_0)[\varphi_0,\varphi_0]\big],\varphi^{\ast}_0\big\rangle,\\
{D}^3_{r\lambda\lambda}\hat{\Phi}(0,\kappa_0,\lambda_0)&=\big\langle-D_{x\lambda\lambda}^3F(0,\lambda_0)[\varphi_0]-2D_{x\lambda}^2F(0,\lambda_0)(i\kappa_0 I-A_0)^{-1}\big(D_{x\lambda}^2F(0,\lambda_0)[\varphi_0]\\
&\quad -\langle D_{x\lambda}^2F(0,\lambda_0)[\varphi_0],\varphi^{\ast}_0\rangle\varphi_0\big),\varphi^{\ast}_0\big\rangle.
\end{align*}
Moreover,
$$H_{11}=\frac{1}{3}\rea{D}^3_{rrr}\hat{\Phi}(0,\kappa_0,\lambda_0)\quad \hbox{and }\quad H_{22}=\rea{D}^3_{r\lambda\lambda}\hat{\Phi}(0,\kappa_0,\lambda_0). $$
\end{lem}

\begin{proof}[{\bf Proof. }]
In view of \eqref{eq:vv}, write $Px(t)=r(\psi_{0}+\overline{\psi}_{0})=r{\hat{v}_{2}}(t)$.
From \eqref{b4}, we obtain that
\begin{equation}\label{eq:psifirst}
{D}_r\hat{\Phi}(r,\kappa,\lambda)=\frac{1}{2\pi}\int_0^{2\pi}\langle D_xG(r{\hat{v}_{2}}+\psi(r{\hat{v}_{2}},\kappa,\lambda),\kappa,\lambda)\big[{\hat{v}_{2}}+D_v\psi[{\hat{v}_{2}}]\big],\psi^{\ast}_0\rangle  \mathrm{d} t.
\end{equation}
Differentiating ${D}_r\hat{\Phi}$ with respect to $r,\kappa$ and $\lambda$ yield
\begin{equation}
\begin{split}
{D}^2_{rr}\hat{\Phi}(r,\kappa,\lambda)=&\frac{1}{2\pi}\int_0^{2\pi}\langle D_{xx}^2G(r{\hat{v}_{2}}+\psi(r{\hat{v}_{2}},\kappa,\lambda),\kappa,\lambda)\big[\hat{v}_{2}+D_v\psi[\hat{v}_{2}],\hat{v}_{2}+D_v\psi[\hat{v}_{2}]\big]\\
&+D_xG(r{\hat{v}_{2}}+\psi(r{\hat{v}_{2}},\kappa,\lambda),\kappa,\lambda)\big[D_{vv}^2\psi[\hat{v}_{2},\hat{v}_{2}]\big],\psi^{\ast}_0\rangle  \mathrm{d} t,\\
{D}^2_{r\kappa}\hat{\Phi}(r,\kappa,\lambda)=&\frac{1}{2\pi}\int_0^{2\pi}\langle D_{xx}^2G(r{\hat{v}_{2}}+\psi(r{\hat{v}_{2}},\kappa,\lambda),\kappa,\lambda)\big[D_{\kappa}\psi,{\hat{v}_{2}}+D_v\psi[{\hat{v}_{2}}]\big]\\
&+D_{x\kappa}^2G(r{\hat{v}_{2}}+\psi(r{\hat{v}_{2}},\kappa,\lambda),\kappa,\lambda)\big[{\hat{v}_{2}}+D_v\psi[{\hat{v}_{2}}]\big]\\
&+D_xG(r{\hat{v}_{2}}+\psi(r{\hat{v}_{2}},\kappa,\lambda),\kappa,\lambda)\big[D_{v\kappa}^2\psi[{\hat{v}_{2}}]\big],\psi^{\ast}_0\rangle  \mathrm{d} t,\\
{D}^2_{r\lambda}\hat{\Phi}(r,\kappa,\lambda)=&\frac{1}{2\pi}\int_0^{2\pi}\langle D_{xx}^2G(r{\hat{v}_{2}}+\psi(r{\hat{v}_{2}},\kappa,\lambda),\kappa,\lambda)\big[D_{\lambda}\psi,{\hat{v}_{2}}+D_v\psi[{\hat{v}_{2}}]\big]\\
&+D_{x\lambda}^2G(r{\hat{v}_{2}}+\psi(r{\hat{v}_{2}},\kappa,\lambda),\kappa,\lambda)\big[{\hat{v}_{2}}+D_v\psi[{\hat{v}_{2}}]\big]\\
&+D_xG(r{\hat{v}_{2}}+\psi(r{\hat{v}_{2}},\kappa,\lambda),\kappa,\lambda)\big[D_{v\lambda}^2\psi[{\hat{v}_{2}}]\big],\psi^{\ast}_0\rangle  \mathrm{d} t.\label{b6}\\
\end{split}
\end{equation}

Since
\begin{equation}\label{eq:equality}
\int_0^{2\pi} e^{nit} \mathrm{d} t=0\  (n\ne 0)  \hbox {  and  } \int_0^{2\pi}\langle x,\psi^{\ast}_0\rangle  \mathrm{d} t=0\quad
\hbox{for } x\in R(I-Q)=R(J_0),
\end{equation}
it follows from \eqref{eq:psi0} and \eqref{b6}${}_1$ that $${D}^2_{rr}\hat{\Phi}(0,\kappa_0,\lambda_0)=0.$$
Moreover,  it follows from \eqref{eq:psi0} and \eqref{b6}${}_2$ that
\begin{equation*}
\begin{split}
{D}^2_{r\kappa}\hat{\Phi}(0,\kappa_0,\lambda_0)&=\frac{1}{2\pi}\int_0^{2\pi}\langle D_{x\kappa}^2G(0,\kappa_0,\lambda_0)[{\hat{v}_{2}}],\psi^{\ast}_0\rangle  \mathrm{d} t\\
&=\frac{1}{2\pi}\int_0^{2\pi}\langle \frac{\mathrm{d}\hat{v}_{2}}{ \mathrm{d} t},\psi^{\ast}_0\rangle  \mathrm{d} t
=\frac{1}{2\pi}\int_0^{2\pi}\langle i\varphi_0,\varphi^{\ast}_0\rangle  \mathrm{d} t=i\langle \varphi_0,\varphi^{\ast}_0\rangle=i.
\end{split}
\end{equation*}
Also,  we obtain from \eqref{eq:psi0} and  \eqref{b6}${}_3$ that
\begin{equation*}
\begin{split}
{D}^2_{r\lambda}\hat{\Phi}(0,\kappa_0,\lambda_0)&=\frac{1}{2\pi}\int_0^{2\pi}\langle D_{x\lambda}^2G(0,\kappa_0,\lambda_0)[{\hat{v}_{2}}],\psi^{\ast}_0\rangle  \mathrm{d} t\\
&=\frac{1}{2\pi}\int_0^{2\pi}\langle -D_{x\lambda}^2F(0,\lambda_0)[{\hat{v}_{2}}],\psi^{\ast}_0\rangle  \mathrm{d} t=-\langle D_{x\lambda}^2F(0,\lambda_0)[\varphi_0],\varphi^{\ast}_0\rangle.
\end{split}
\end{equation*}

Before proceeding to the 3rd derivatives of $\hat{\Phi}$,
we need to calculate the 2nd derivatives of $\psi$ at $(0,\kappa_0,\lambda_0)$.
With \eqref{b1} and \eqref{eq:redu} in mind,
differentiating the equation $(I-Q)G(v+\psi(v,\kappa,\lambda),\kappa,\lambda)=0$ with respect to $v$ once and twice yield
\begin{align}
&(I-Q)D_xG(v+\psi(v,\kappa,\lambda),\kappa,\lambda)[I+D_v\psi]=0,\label{b9}\\\nonumber
&(I-Q)D_{xx}^2G(v+\psi(v,\kappa,\lambda),\kappa,\lambda)[I+D_v\psi,I+D_v\psi]\\
&\qquad\qquad\qquad\qquad\qquad+(I-Q)D_xG(v+\psi(v,\kappa,\lambda),\kappa,\lambda)[D_{vv}^2\psi]=0.\label{b10}
\end{align}
By \eqref{eq:psi0}, acting \eqref{b10} on $[\hat{v}_2,\hat{v}_2]$ at $(0,\kappa_0,\lambda_0)$ gives
\begin{equation*}
(I-Q)D_{xx}^2G(0,\kappa_0,\lambda_0)[\hat{v}_{2},\hat{v}_{2}]+(I-Q)J_0\big[D_{vv}^2\psi(0,\kappa_0,\lambda_0)[\hat{v}_2,\hat{v}_2]\big]=0.
\end{equation*}
Then we have
\begin{equation}
\begin{split}
J_0\big[D_{vv}^2\psi(0,\kappa_0,\lambda_0)[\hat{v}_{2},\hat{v}_{2}]\big]    &  =-(I-Q)D_{xx}^2G(0,\kappa_0,\lambda_0)[\hat{v}_{2},\hat{v}_{2}]\\
& =(I-Q)D_{xx}^2F(0,\kappa_0,\lambda_0)[\hat{v}_{2},\hat{v}_{2}].
\end{split}\label{b11}
\end{equation}
Since $z_0e^{\pm nit}\in R(J_0)$, solving $J_0x=z_0e^{\pm nit}$  for $x\in (I-P)(E \cap Y)$ yields
\begin{equation}
x=(\pm ni\kappa_0I-A_0)^{-1}z_0e^{\pm nit}\qquad \text{for } n\neq 1.\label{b12}
\end{equation}
Then \eqref{b11} and \eqref{b12} imply
\begin{equation}\label{b12.1}
\begin{split}
D_{vv}^2\psi(0,\kappa_0,&\lambda_0)[{\hat{v}_{2}},{\hat{v}_{2}}]=J_0^{-1}\big[D_{xx}^2F(0,\lambda_0)[\varphi_0,\varphi_0]e^{2it}\big]+2J_0^{-1}\big[D_{xx}^2F(0,\lambda_0)[\varphi_0,\overline{\varphi}_0]\big]\\
&\qquad\qquad\quad+J_0^{-1}\big[D_{xx}^2F(0,\lambda_0)[\overline{\varphi}_0,\overline{\varphi}_0]e^{-2it}\big]\\
=&(2i\kappa_0 I-A_0)^{-1}D_{xx}^2F(0,\lambda_0)[\varphi_0,\varphi_0]e^{2it}+2(-A_0)^{-1}D_{xx}^2F(0,\lambda_0)[\varphi_0,\overline{\varphi}_0]\\
&+(-2i\kappa_0 I-A_0)^{-1}D_{xx}^2F(0,\kappa_0)[\overline{\varphi}_0,\overline{\varphi}_0]e^{-2it}.
\end{split}
\end{equation}

Also, differentiating \eqref{b9} in $\kappa$ yields
\begin{equation*}
\begin{split}
&(I-Q)\Big\{D_{xx}^2G(v+\psi(v,\kappa,\lambda),\kappa,\lambda)[D_{\kappa}\psi,I+D_v\psi]+D_{x\kappa}^2G(v+\psi(v,\kappa,\lambda),\kappa,\lambda)[I+D_v\psi]
\\
&\qquad +D_xG(v+\psi(v,\kappa,\lambda),\kappa,\lambda)[D_{v\kappa}^2\psi]\Big\}=0.
\end{split}
\end{equation*}
By \eqref{eq:psi0}, acting the above equation on $\hat{v}_2$ at $(0,\kappa_0,\lambda_0)$ gives
\begin{equation*}
(I-Q)D_{x\kappa}^2G(0,\kappa_0,\lambda_0)[{\hat{v}_{2}}]+J_0[D_{v\kappa}^2\psi(0,\kappa_0,\lambda_0)[{\hat{v}_{2}}]]=0.
\end{equation*}
Since $\frac{\mathrm{d}\hat{v}_{2}}{ \mathrm{d} t}=i(\varphi_0e^{it}-\overline{\varphi}_0e^{-it})\in N(J_0)$ and $I-Q$ is the projection to $R(J_0)$ along $N(J_0)$, we have
\begin{equation}\label{b16}
J_0\big[D_{v\kappa}^2\psi(0,\kappa_0,\lambda_0)[{\hat{v}_{2}}]\big]=-(I-Q)D_{x\kappa}^2G(0,\kappa_0,\lambda_0)[{\hat{v}_{2}}]=-(I-Q)\frac{\mathrm{d}\hat{v}_{2}}{ \mathrm{d} t}=0.
\end{equation}
Noticing that $D_{v\kappa}^2\psi[{\hat{v}_{2}}]\in (I-P)(E \cap Y)$, we get
\begin{equation}
D_{v\kappa}^2\psi(0,\kappa_0,\lambda_0)[{\hat{v}_{2}}]=0.\label{b17}
\end{equation}

Similarly, differentiating \eqref{b9} in $\lambda$ yields
\begin{equation*}
\begin{split}
(I-Q)&\Big\{D_{xx}^2G(v+\psi(v,\kappa,\lambda),\kappa,\lambda)[D_{\lambda}\psi,I+D_v\psi]\\
&+D_{x\lambda}^2G(v+\psi(v,\kappa,\lambda),\kappa,\lambda)[I+D_v\psi] +D_{x}G(v+\psi(v,\kappa,\lambda),\kappa,\lambda)[D_{v\lambda}^2\psi]\Big\}=0.
\end{split}
\end{equation*}
By \eqref{eq:psi0}, acting  the above equation on $\hat{v}_2$ at $(0,\kappa_0,\lambda_0)$ gives
\begin{equation*}
(I-Q)D_{x\lambda}^2G(0,\kappa_0,\lambda_0)[{\hat{v}_{2}}]+J_0\big[D_{v\lambda}^2\psi(0,\kappa_0,\lambda_0)[{\hat{v}_{2}}]\big]=0.
\end{equation*}
Thus,
\begin{equation}
J_0\big[D_{v\lambda}^2\psi(0,\kappa_0,\lambda_0)[{\hat{v}_{2}}]\big]=-(I-Q)D_{x\lambda}^2G(0,\kappa_0,\lambda_0)[{\hat{v}_{2}}]=(I-Q)D_{x\lambda}^2F(0,\lambda_0)[{\hat{v}_{2}}].\label{b18}
\end{equation}
It follows from \eqref{eq:Q} that
\begin{equation*}
\begin{split}
&Q(D_{x\lambda}^2F(0,\lambda_0)[\varphi_0e^{it}])\\
=&\frac{1}{2\pi}\Big(\int_0^{2\pi}\!\langle D_{x\lambda}^2F[\varphi_0e^{it}],\varphi^{\ast}_0e^{-it} \mathrm{d} t\rangle\Big) \varphi_0e^{it}+\frac{1}{2\pi}\Big(\int_0^{2\pi}\!\langle D_{x\lambda}^2F[\varphi_0e^{it}],\overline{\varphi}^{\ast}_0e^{it}\rangle  \mathrm{d} t\Big)\overline{\varphi}_0e^{-it}\\
=&\langle D_{x\lambda}^2F[\varphi_0],\varphi^{\ast}_0\rangle \varphi_0e^{it},
\end{split}
\end{equation*}
then
\begin{equation}
\begin{split}
(I-Q)D_{x\lambda}^2F(0,\lambda_0)[\varphi_0e^{it}]&=D_{x\lambda}^2F(0,\lambda_0)[\varphi_0]e^{it}-\langle D_{x\lambda}^2F(0,\lambda_0)[\varphi_0],\varphi^{\ast}_0\rangle \varphi_0e^{it}\\
&=\big(D_{x\lambda}^2F(0,\lambda_0)[\varphi_0]-\langle D_{x\lambda}^2F(0,\lambda_0)[\varphi_0],\varphi^{\ast}_0\rangle \varphi_0\big)e^{it}.\label{b19}
\end{split}
\end{equation}
%
Solving $J_0x=(I-Q)D_{x\lambda}^2F(0,\lambda_0)[\varphi_0]e^{it}$ for $x\in (I-P)(E \cap Y)$ gives
\begin{equation}
x=(i\kappa_0 I-A_0)^{-1}\left(D_{x\lambda}^2F(0,\lambda_0)[\varphi_0]-\langle D_{x\lambda}^2F(0,\lambda_0)[\varphi_0],\varphi_0^*\rangle\varphi_0\right)e^{it},\label{b21}
\end{equation}
and it is analogous when replacing $\varphi_0 e^{it}$ with $\overline{\varphi}_0 e^{-it}$.
From \eqref{b18}--\eqref{b21}, it follows that
\begin{equation}
\begin{split}
D_{v\lambda}^2\psi(0,\kappa_0,\lambda_0)[&{\hat{v}_{2}}]=(i\kappa_0 I-A_0)^{-1}\left(D_{x\lambda}^2F(0,\lambda_0)[\varphi_0]-\langle D_{x\lambda}^2F(0,\lambda_0)[\varphi_0],\varphi^{\ast}_0\rangle\varphi_0\right)e^{it}\\
&+(-i\kappa_0 I-A_0)^{-1}\left(D_{x\lambda}^2F(0,\lambda_0)[\overline{\varphi}_0]-\langle D_{x\lambda}^2F(0\lambda_0)[\overline{\varphi}_0],\varphi^{\ast}_0\rangle\overline{\varphi}_0\right)e^{-it}.\label{b22}
\end{split}
\end{equation}
Here the equality holds because the right hand side of \eqref{b22} is 
the solution of $J_0[x]=(I-Q)D_{x\lambda}^2F(0,\lambda_0)[{\hat{v}_{2}}]$ in $(I-P)(E\cap Y)$, which is the complement of $N(J_0)$, and so does the left hand side of \eqref{b22}.


We now continue to calculate the 3rd derivatives of $\hat{\Phi}$ at $(0,\kappa_0,\lambda_0)$. 
Differentiating \eqref{b6}$_1$ in $r$ at $(0,\kappa_0,\lambda_0)$ yields
\begin{align*}
{D}^3_{rrr}\hat{\Phi}(0,\kappa_0,\lambda_0)
=&\frac{1}{2\pi}\int_0^{2\pi}\langle D_{xxx}^3G^0[\hat{v}_{2},\hat{v}_{2},\hat{v}_{2}]
+3D_{xx}^2G^0\big[{\hat{v}_{2}},D_{vv}^2\psi^0[\hat{v}_{2},\hat{v}_{2}] \big]\\
&+J_0\big[D_{vvv}^3\psi^0[\hat{v}_{2},\hat{v}_{2},\hat{v}_{2}]\big],\varphi^{\ast}_0e^{-it}\rangle  \,\mathrm{d} t.\label{p1}
\end{align*}
Here and in what follows, the superscript `${}^0$' means the evaluation of the function at $(0,\kappa_0,\lambda_0)$ for brevity.
Using \eqref{eq:equality} and \eqref{b12.1}, we further get that
\begin{equation*}
\begin{split}
{D}^3_{rrr}\hat{\Phi}&(0,\kappa_0,\lambda_0)=\frac{1}{2\pi}\int_0^{2\pi}\big\langle -3D_{xxx}^3F(0,\lambda_0)[\varphi_0,\varphi_0,\overline{\varphi}_0]\\
&\qquad\qquad\quad-3D_{xx}^2F(0,\lambda_0)\big[\overline{\varphi}_0,(2i\kappa_0 I-A_0)^{-1}D_{xx}^2F(0,\lambda_0)[\varphi_0,\varphi_0]\big]\\
&\qquad\qquad\quad-6D_{xx}^2F(0,\lambda_0)\big[\varphi_0,(-A_0)^{-1}D_{xx}^2F(0,\lambda_0)[\varphi_0,\overline{\varphi}_0]\big],\varphi^{\ast}_0\big\rangle  \,\mathrm{d} t\\
=&\big\langle-3D_{xxx}^3F(0,\lambda_0)[\varphi_0,\varphi_0,\overline{\varphi}_0]
-3D_{xx}^2F(0,\lambda_0)\big[\overline{\varphi}_0,(2i\kappa I-A_0)^{-1}D_{xx}^2F(0,\lambda_0)[\varphi_0,\varphi_0]\big]\\
&\qquad\qquad\quad-6D_{xx}^2F(0,\lambda_0)\big[\varphi_0,(-A_0)^{-1}D_{xx}^2F(0,\lambda_0)[\varphi_0,\overline{\varphi}_0]\big],\varphi^{\ast}_0\big\rangle.\label{p1.1}
\end{split}
\end{equation*}

Differentiating \eqref{b6}$_1$ in $\kappa$ at $(0,\kappa_0,\lambda_0)$ yields
\begin{equation*}
\begin{split}
{D}^3_{rr\kappa}\hat{\Phi}(0,\kappa_0,\lambda_0)=&\frac{1}{2\pi}\int_0^{2\pi}\big\langle D_{xxx}^3G^0[D_{\kappa}\psi^0,{\hat{v}_{2}},{\hat{v}_{2}}]+D_{xx\kappa}^3G^0[\hat{v}_{2},\hat{v}_{2}]\\
&+2D_{xx}^2G^0\big[{\hat{v}_{2}},D_{v\kappa}^2\psi^0[{\hat{v}_{2}}]\big]+D_{xx}^2G^0\big[D_{\kappa}\psi^0,D_{vv}^2\psi^0[\hat{v}_{2},\hat{v}_{2}]\big]\\
&+D_{x\kappa}^2G^0\big[D_{vv}^2\psi^0[\hat{v}_{2},\hat{v}_{2}]\big]+J_0\big[D_{vv\kappa}^3\psi^0[\hat{v}_{2},\hat{v}_{2}]\big],\varphi^{\ast}_0e^{-it}\big\rangle  \,\mathrm{d} t.\label{p2}
\end{split}
\end{equation*}
Since $D_{xx\kappa}^3G(0,\kappa_0,\lambda_0)= 0$, it follows from \eqref{eq:psi0}, \eqref{eq:equality}, \eqref{b12.1} and \eqref{b17} that
\begin{equation}
{D}^3_{rr\kappa}\hat{\Phi}(0,\kappa_0,\lambda_0)=0.\label{p2.1}
\end{equation}

Differentiating \eqref{b6}$_1$ in $\lambda$ at $(0,\kappa_0,\lambda_0)$ yields
\begin{equation*}
\begin{split}
{D}^3_{rr\lambda}\hat{\Phi}(0,\kappa_0,\lambda_0)=&\frac{1}{2\pi}\int_0^{2\pi}\big\langle D_{xxx}^3G^0[D_{\lambda}\psi^0,{\hat{v}_{2}},{\hat{v}_{2}}]+D_{xx\lambda}^3G^0[\hat{v}_{2},\hat{v}_{2}]\\
&+2D_{xx}^2G^0\big[{\hat{v}_{2}},D_{v\lambda}^2\psi^0[{\hat{v}_{2}}]\big]+D_{xx}^2G^0\big[D_{\lambda}\psi^0,D_{vv}^2\psi^0[\hat{v}_{2},\hat{v}_{2}]\big]\\
&+D_{x\lambda}^2G^0\big[D_{vv}^2\psi^0[\hat{v}_{2},\hat{v}_{2}]\big]+J_0\big[D_{vv\lambda}^3\psi^0[\hat{v}_{2},\hat{v}_{2}]\big],\varphi^{\ast}_0e^{-it}\big\rangle  \,\mathrm{d} t.\label{p3}
\end{split}
\end{equation*}
Combining \eqref{eq:psi0}, \eqref{eq:equality}, \eqref{b12.1} and \eqref{b22},
we obtain from the above equation that
\begin{equation*}
{D}^3_{rr\lambda}\hat{\Phi}(0,\kappa_0,\lambda_0)=0.\label{p3.1}
\end{equation*}

Differentiating \eqref{b6}$_2$ in $\lambda$ at $(0,\kappa_0,\lambda_0)$, we obtain from \eqref{eq:psi0}, \eqref{b17} and $D_{x\kappa\lambda}^2G^0=0$ that
\begin{equation*}
{D}_{r\kappa\lambda}^3\hat{\Phi}(0,\kappa_0,\lambda_0)=\frac{1}{2\pi}\int_0^{2\pi}\langle D_{x\kappa}^2G^0\big[D_{v\lambda}^2\psi^0[\hat{v}_2]\big]+J_0\big[D_{v\kappa\lambda}^3\psi^0[\hat{v}_2]\big],\varphi^{\ast}_0e^{-it}\rangle \mathrm{d}t.
\end{equation*}
Since $D_{v\lambda}^2\psi^0[\hat{v}_2]\in R(J_0)$ and hence $\frac{\mathrm{d}}{\mathrm{d}t}D_{v\lambda}^2\psi^0[\hat{v}_2]\in R(J_0)$, \eqref{eq:equality} implies
\begin{equation*}
{D}^3_{r\kappa\lambda}\hat{\Phi}(0,\kappa_0,\lambda_0)=\frac{1}{2\pi}\int_0^{2\pi}\langle \frac{\mathrm{d}}{\mathrm{d}t}D_{v\lambda}^2\psi^0[\hat{v}_2]+J_0\big[D_{v\kappa\lambda}^3\psi^0[\hat{v}_2]\big],\varphi^{\ast}_0e^{-it} \rangle\,\mathrm{d}t=0.
\end{equation*}
Similarly, differentiating \eqref{b6}$_2$ in $\kappa$ at $(0,\kappa_0,\lambda_0)$ gives
\begin{equation*}
{D}^3_{r\kappa\kappa}\hat{\Phi}(0,\kappa_0,\lambda_0)=0
\end{equation*}

Differentiating \eqref{b6}$_3$ in $\lambda$ at $(0,\kappa_0,\lambda_0)$ yields
\begin{equation*}
\begin{split}\label{p6}
{D}^3_{r\lambda\lambda}\hat{\Phi}(0,\kappa_0,\lambda_0)=&\frac{1}{2\pi}\int_0^{2\pi}\big\langle \frac{\mathrm{d}}{\mathrm{d}\lambda}D_{xx}^2G(r\hat{v}_2
+\psi(r\hat{v}_2,\kappa,\lambda),\kappa,\lambda)\big|_{(r,\kappa,\lambda)
=(0,\kappa_0,\lambda_0)}[D_{\lambda}\psi^0,{\hat{v}_{2}}]\\
&+D_{xx}^2G^0[D_{\lambda\lambda}^2\psi^0,{\hat{v}_{2}}]
+2D_{xx}^2G^0\big[D_{\lambda}\psi^0,D_{v\lambda}^2\psi^0[{\hat{v}_{2}}]\big]
+D_{xx\lambda}^3G^0[D_{\lambda}\psi^0,\hat{v}_2]\\
&+D_{x\lambda\lambda}^3G^0[{\hat{v}_{2}}]
+2D_{x\lambda}^2G^0\big[D_{v\lambda}^2\psi^0[{\hat{v}_{2}}]\big]
+J_0\big[D_{v\lambda\lambda}^3\psi^0[{\hat{v}_{2}}]\big],\varphi^{\ast}_0e^{-it}\big\rangle  \,\mathrm{d} t.
\end{split}
\end{equation*}
From \eqref{eq:psi0}, \eqref{b22} and \eqref{eq:equality}, we obtain that
\begin{equation*}
\begin{split}
{D}^3_{r\lambda\lambda}\hat{\Phi}&(0,\kappa_0,\lambda_0)=\langle-D_{x\lambda\lambda}^3F(0,\lambda_0)[\varphi_0],\varphi^{\ast}_0\rangle\\
&+2\langle-D_{x\lambda}^2F(0,\lambda_0)(i\kappa_0 I-A_0)^{-1}\big(D_{x\lambda}^2F(0,\lambda_0)[\varphi_0]-\langle D_{x\lambda}^2F(0,\lambda_0)[\varphi_0],\varphi^{\ast}_0\rangle\varphi_0\big),\varphi^{\ast}_0\rangle,
\end{split}
\end{equation*}
which completes the proof.
\end{proof}

\subsection{Proof}

\begin{proof}[\textbf{Proof of Theorem \ref{thm2}.}]
To use Lemma \ref{lem:morse} to prove the theorem, we need to
reduce three variables ``$r,\kappa,\lambda$'' in equation \eqref{b0} to two.

From \eqref{b0}, we obtain that
\begin{equation}
\begin{split}
&D_r\tilde{\Phi}(r,\kappa,\lambda)=\int_0^1{D}^2_{rr}\hat{\Phi}(\tau r,\kappa,\lambda)\tau \,\mathrm{d} \tau,\\
&D_\kappa\tilde{\Phi}(r,\kappa,\lambda)=\int_0^1{D}^2_{r\kappa}\hat{\Phi}(\tau r,\kappa,\lambda)\,\mathrm{d} \tau,\\
&D_\lambda\tilde{\Phi}(r,\kappa,\lambda)=\int_0^1{D}^2_{r\lambda}\hat{\Phi}(\tau r,\kappa,\lambda)\,\mathrm{d} \tau.\label{b5}
\end{split}
\end{equation}
By Lemma \ref{lem:3.1}, we have
\begin{equation}
\begin{split}
D_r\tilde{\Phi}(0,\kappa_0,\lambda_0)&=\int_0^1{D}^2_{rr}\hat{\Phi}(0,\kappa_0,\lambda_0)\tau \,\mathrm{d} \tau=0,\\
D_\kappa\tilde{\Phi}(0,\kappa_0,\lambda_0)&=\int_0^1{D}^2_{r\kappa}\hat{\Phi}(0,\kappa_0,\lambda_0)\,\mathrm{d} \tau=i,\\
D_\lambda\tilde{\Phi}(0,\kappa_0,\lambda_0)&=\int_0^1{D}^2_{r\lambda}\hat{\Phi}(0,\kappa_0,\lambda_0)\,\mathrm{d} \tau=-\langle D_{x\lambda}^2F(0,\lambda_0)[\varphi_0],\varphi^{\ast}_0\rangle.\label{b7}
\end{split}
\end{equation}

Set
\begin{equation*}\label{b25}
g(r,\kappa,\lambda)=\rea \tilde{\Phi}(r,\kappa,\lambda),\quad\tilde{g}(r,\kappa,\lambda)=\im \tilde{\Phi}(r,\kappa,\lambda).
\end{equation*}
Clearly,
\begin{equation}\label{b27}
\tilde{\Phi}(r,\kappa,\lambda)=0\Leftrightarrow g(r,\kappa,\lambda)=0,\;\tilde{g}(r,\kappa,\lambda)=0.
\end{equation}
Then $\eqref{b7}_2$ implies that
\begin{equation}\label{eq:implicit}
D_{\kappa}\tilde{g}(0,\kappa_0,\lambda_0)=\im D_\kappa\tilde{\Phi}(0,\kappa_0,\lambda_0)=1\ne 0.
\end{equation}
Moreover, we obtain that $\tilde{g}(0,\kappa_0,\lambda_0)=g(0,\kappa_0,\lambda_0)=0$ by $\tilde{\Phi}(0,\kappa_0,\lambda_0)=0$.
So \eqref{eq:implicit} ensures that the equation $\tilde{g}(r,\kappa,\lambda)=0$ is solvable for $r\in(-\delta,\delta)$ and $\lambda\in(\lambda_0-\delta,\lambda_0+\delta)$ by the Implicit Function Theorem, which yields a continuously differentiable function $\kappa(r,\lambda)$ such that
\begin{equation}\label{kappa}
\tilde{g}(r,\kappa(r,\lambda),\lambda)=0\text{ and }\kappa(0,\lambda_0)=\kappa_0.
\end{equation}
Thus, we view $g$ as a function of two variables ``$r,\lambda$'' and set
 $$f(r,\lambda)\equiv g(r,\kappa(r,\lambda),\lambda).$$
Then $f(0,\lambda_0)=g(0,\kappa_0,\lambda_0)=0$. 

To solve the left equation in \eqref{b27}, it suffices to solve the equation
\begin{equation}
f(r,\lambda)=0.\label{b28}
\end{equation}
Then \eqref{b7} implies that
\begin{align}\nonumber
{D} \kappa(0,\lambda_0)&=\big(D_{r}\kappa (0,\lambda_0),D_{\lambda}\kappa (0,\lambda_0)\big{)}=\Big(-\frac{D_{r}\tilde g ^0}{D_{\kappa}\tilde g ^0},-\frac{D_{\lambda}\tilde g ^0}{D_{\kappa}\tilde g ^0}\Big)\\\label{b30}
&=\Big(-\frac{\im D_r\tilde{\Phi}^0}{\im D_\kappa\tilde{\Phi}^0},-\frac{\im D_\lambda\tilde{\Phi}^0}{\im D_\kappa\tilde{\Phi}^0}\Big) =\big(0,\im \langle D_{x\lambda}^2F(0,\lambda_0)[\varphi_0],\varphi^{\ast}_0\rangle\big).
\end{align}
Differentiating \eqref{b28} in $r$ and $\lambda$ yield
\begin{align}\nonumber
D_{r} f (r,\lambda)&=D_{r} g (r,\kappa,\lambda)+D_{\kappa} g (r,\kappa,\lambda)D_{r}\kappa(r,\lambda) \\\label{b31}
&=\rea D_r\tilde{\Phi}(r,\kappa,\lambda)+\rea D_\kappa\tilde{\Phi}(r,\kappa,\lambda)D_{r}\kappa(r,\lambda),\\\nonumber
D_{\lambda} f (r,\lambda)&=D_{\kappa} g (r,\kappa,\lambda)D_{\lambda}\kappa(r,\lambda) +D_{\lambda} g (r,\kappa,\lambda)\\\label{b32}
&=\rea D_\kappa\tilde{\Phi}(r,\kappa,\lambda)D_{\lambda}\kappa(r,\lambda) +\rea D_\lambda\tilde{\Phi}(r,\kappa,\lambda).
\end{align}
The above expressions at $(0,\lambda_0)$ become
\begin{equation}
\begin{split}
&D_{r} f (0,\lambda_0)=0\text{ and } D_{\lambda} f (0,\lambda_0)=-\rea \langle D_{x\lambda}^2F(0,\lambda_0)[\varphi_0],\varphi^{\ast}_0\rangle.\label{b33}
\end{split}
\end{equation}

If the degeneracy condition (F4$^{\prime}$) holds, i.e. 
\begin{equation}
\rea \langle D_{x\lambda}^2F(0,\lambda_0)[\varphi_0],\varphi^{\ast}_0\rangle=0,\label{t1}
\end{equation}
then the gradient of $f$ at $(0,\lambda_0)$ becomes
\begin{equation}
{D}  f(0,\lambda_0)=(0,0).\label{b34}
\end{equation}
We have to calculate the Hessian matrix $H={D} ^2f(0,\lambda_0)$.

Differentiating \eqref{b31} in $r$,
from \eqref{b5}, \eqref{b7}, \eqref{b30} and  Lemma \ref{lem:3.1}, we obtain
\begin{equation*}
\begin{split}
D_{rr}^2f(0,\lambda_0)
=&\rea D_{rr}\tilde{\Phi}^0 +2\rea D_{r\kappa}^2\tilde{\Phi} ^0 D_{r}\kappa(0,\lambda_0) +\rea D_{\kappa\kappa}^2\tilde{\Phi} ^0 \big(D_{r}\kappa (0,\lambda_0)\big)^2\\
&+\rea D_\kappa\tilde{\Phi}^0 D_{rr}\kappa(0,\lambda_0) \\
=&\rea D_{rr}\tilde{\Phi}^0
=\frac{1}{3}\rea {D}^3_{rrr}\hat{\Phi}^0
=H_{11}.\label{b35}
\end{split}
\end{equation*}
Similarly, differentiating \eqref{b31} in $\lambda$,  we obtain from \eqref{b5}, \eqref{b7}, \eqref{b30} and Lemma \ref{lem:3.1} that
\begin{equation*}
\begin{split}
D_{r\lambda}^2f (0,\lambda_0)
=&\rea D_{r\kappa}^2\tilde{\Phi} ^0 D_{\lambda}\kappa(0,\lambda_0) +\rea D_{r\lambda}^2\tilde{\Phi} ^0 +\rea D_{\kappa\kappa}^2\tilde{\Phi} ^0 D_{r}\kappa(0,\lambda_0) D_{\lambda}\kappa(0,\lambda_0)\\
& +\rea D_{\kappa\lambda}^2\tilde{\Phi} ^0 D_{r}\kappa(0,\lambda_0) +\rea D_\kappa\tilde{\Phi}^0 D_{r\lambda}^2\kappa(0,\lambda_0) \\
=&\rea D_{r\kappa}^2\tilde{\Phi} ^0 D_{\lambda}\kappa(0,\lambda_0) +\rea D_{r\lambda}^2\tilde{\Phi} ^0\\
=&\frac{1}{2}\rea {D}^3_{rr\kappa}\hat{\Phi}^0 D_{\lambda}\kappa(0,\lambda_0) +\frac{1}{2}\rea {D}^3_{rr\lambda}\hat{\Phi}^0
=0.\label{b36}
\end{split}
\end{equation*}
From \eqref{b7}, \eqref{b30} and Lemma \ref{lem:3.1} again, we have
\begin{equation*}
\begin{split}
D_{\lambda\lambda}^2f (0,\lambda_0)=&\rea D_{\kappa\kappa}^2\tilde{\Phi} ^0 \big(D_{\lambda}\kappa (0,\lambda_0)\big)^2+2\rea D_{\kappa\lambda}^2\tilde{\Phi} ^0 D_{\lambda}\kappa(0,\lambda_0)\\
&  +\rea \tilde{\Phi}_{\kappa}^0 D_{\lambda\lambda}^2\kappa(0,\lambda_0)+\rea D_{\lambda\lambda}^2\tilde{\Phi}^0 \\
=&\rea {D}^3_{r\kappa\kappa}\hat{\Phi}^0 \big(D_{\lambda}\kappa (0,\lambda_0)\big)^2+\rea {D}^3_{r\kappa\lambda}\hat{\Phi}^0  D_{\lambda}\kappa(0,\lambda_0) +\rea {D}^3_{r\lambda\lambda}\hat{\Phi}^0\\
=&\rea {D}^3_{r\lambda\lambda}\hat{\Phi}^0
=H_{22}.\label{b37}
\end{split}
\end{equation*}

In summary,
the above calculations yield
 $$H={D} ^2f(0,\lambda_0)=\begin{pmatrix}
D_{rr}^2f(0,\lambda_0) &D_{r\lambda}^2f(0,\lambda_0) \\D_{r\lambda}^2f(0,\lambda_0) &D_{\lambda\lambda}^2f(0,\lambda_0)
\end{pmatrix}
=
\begin{pmatrix}
H_{11}&0\\0&H_{22}
\end{pmatrix},$$
which implies that $\det H=H_{11}H_{22}.$

If $\det H=H_{11}H_{22}>0$, then Lemma \ref{lem:morse}(1) implies that $(0,\lambda_0)$ is the only solution of \eqref{b28} near $(0,\lambda_0)$, i.e., for the original problem \eqref{a1}, apart from the trivial solution line $\{(0,\lambda)|\lambda\in(\lambda_0-\delta,\lambda_0+\delta)\}$ that is given in condition (F2), there is no other solution near $(0,\lambda_0)$.

If $\det H=H_{11}H_{22}<0$, then Lemma \ref{lem:morse}(2) implies that the set of solutions of \eqref{b28} near $(0,\lambda_0)$ are composed of exactly two solution curves  $\{\left(r_i(s),\lambda_i(s)\right)|s\in(-\delta,\delta), (r_i^{\prime}(0),\lambda_i^{\prime}(0))=(\widetilde{\mu_i},\widetilde{\eta}_i)\}\ (i=1,2)$, 
in which $(\widetilde{\mu_1},\widetilde{\eta}_1)$ and $(\widetilde{\mu_2},\widetilde{\eta}_2)$ are two linear independent solutions of
\begin{equation}
H_{11}\widetilde{\mu}^2+H_{22}\widetilde{\eta}^2=0.\label{b38.1}
\end{equation}
Moreover, $\det H<0$ implies $r_i^{\prime}(0)=\widetilde{\mu}_i\ne 0$ for $i=1,2$. By the Inverse Function Theorem, the solutions $(r_i(s),\lambda_i(s)), i=1,2$ can be reparameterized as continuously differentiable curves $(r,\tilde{\lambda}_i(r)), i=1,2$.
Denote $\tilde{\lambda}$ still by $\lambda$, and $\frac{\mathrm{d}}{\mathrm{d}r}$ by dot `$\cdot$' for simplicity.  So
\begin{equation}\label{eq:lambdaD}
\dot{\lambda}_i(0)=\frac{\mathrm{d}}{\mathrm{d}r}\lambda(r)\big|_{r=0}=\eta_i\neq 0,
\end{equation}
in which $\eta_{1,2}=\pm\sqrt{-{H_{11}}/{H_{22}}}$ are the roots of
$$H_{11}+H_{22}\eta^2=0.$$
Insert $(r,{\lambda}_i(r))$ into \eqref{kappa} and set $\kappa_i(r)=\kappa(r,\lambda_i(r))$. Then from \eqref{b30} we get
\begin{equation}\label{eq:kappaD}
\dot{\kappa}_i(0)=D_{r}\kappa (0,\lambda_0)+D_{\lambda}\kappa (0,\lambda_0)\dot{\lambda}_i(0)=\im \langle D_{x\lambda}^2F(0,\lambda_0)[\varphi_0],\varphi^{\ast}_0\rangle\eta_i,\quad i=1,2.
\end{equation} 
Since $x_i(r)=r\hat{v}_2+\psi(r\hat{v}_2,\kappa_i(r),\lambda_i(r))$, it follows that
\begin{equation}\label{eq:xD}
\dot{x}_i(0)=\hat{v}_2+D_v\psi^0[\hat{v}_2]+D_{\kappa} \psi ^0\dot{\kappa}_i(0)+D_{\lambda} \psi ^0\dot{\lambda}_i(0)=\hat{v}_2=2\rea (\varphi_0 e^{i t}),\quad i=1,2.
\end{equation} 

Furthermore, because of the $S^1$-equivariance of $\hat{\Phi}$ (cf. \cite[(I.8.36)]{Kielhoefer2012}), i.e. $\hat{\Phi}(e^{i\theta}c,\kappa,\lambda)=e^{i\theta}\hat{\Phi}(c,\kappa,\lambda)$, we have $\hat{\Phi}\big(-r,\kappa,\lambda\big)=-\hat{\Phi}\big(r,\kappa,\lambda\big).$ Then
$$\hat{\Phi}\big(r,\kappa_i(-r),\lambda_i(-r)\big)=-\hat{\Phi}\big(-r,\kappa_i(-r),\lambda_i(-r)\big)=0,\quad \hat{\Phi}\big(r,\kappa_j(r),\lambda_j(r)\big)=0.$$ 
Thus, we get three solutions: $(r,\kappa_i(r),\lambda_i(r))$, $(r,\kappa_j(r),\lambda_j(r))$ and $(r,\kappa_i(-r),\lambda_i(-r))$.
But Lemma \ref{lem:morse}(2) implies that there are only two solutions. Hence, two of them must be the same. Since the first two solutions are linearly independent, and
$\lambda_i(-r)\neq\lambda_i(r)$ for small $r$ by $\dot{\lambda}_i(0)\neq 0$, it follows that the latter two must be the same, i.e.,
\begin{equation}\label{eq:change}
\kappa_i(-r)=\kappa_j(r),\;\lambda_i(-r)=\lambda_j(r),\quad i,j=1,2, \quad i\ne j.
\end{equation}
Moreover, from \eqref{eq:change} and the $S^1$-equivariance of $\psi$ (cf. \cite[(I.8.32)]{Kielhoefer2012}), i.e., $S_{\theta}\psi(Px,\kappa,\lambda)=\psi(S_{\theta}Px,\kappa,\lambda)$, we obtain
\begin{equation}\label{eq:changeji}
\begin{split}
S_{\pi} x_j(r)=&S_{\pi}\left(r\hat{v}_2+\psi(r\hat{v}_2,\kappa_j(r),\lambda_j(r))\right)\\
=&S_{\pi}r(\varphi_0 e^{it}+\overline{\varphi}_0 e^{-it})+\psi\left(S_{\pi}r(\varphi_0 e^{it}+\overline{\varphi}_0 e^{-it}),\kappa_j(r),\lambda_j(r)\right)\\
=&-r(\varphi_0 e^{it}+\overline{\varphi}_0 e^{-it})+\psi\left(-r(\varphi_0 e^{it}+\overline{\varphi}_0 e^{-it}),\kappa_i(-r),\lambda_i(-r)\right)\\
=&x_i(-r),\quad i,j=1,2,\;i\neq j.\\
\end{split}
\end{equation}
Thus for $(r,\kappa_i(r),\lambda_i(r))$ and $(r,\kappa_j(r),\lambda_j(r))$,
each of these two local solution curves is actually the phase shift of the other,
which means that they can be regarded as the same because the $S^1$-equivariance of problem \eqref{b1}.

Finally, come back from \eqref{b1} to \eqref{a1} by rescaling the variable $\tilde{t}=t/\kappa(r)$, still denoting by $t$. Then apart from \eqref{eq:change}, we also obtain from \eqref{eq:changeji} that
$$x_i(-r)=S_{\pi / \kappa_j(r)} x_j(r);$$ 
see Fig.\ref{fig:surface}. Regarding $x_i$ and $x_j$ as the same solution and denoting by $x$, we obtain from \eqref{eq:lambdaD}--\eqref{eq:xD} that
$$\dot{x}(0)=2\rea (\varphi_0 e^{i\kappa_0 t}), \quad \dot{\lambda}(0)=\eta=\sqrt{-{H_{11}}/{H_{22}}}, \quad \dot{\kappa}(0)=\im \langle D_{x\lambda}^2F(0,\lambda_0)[\varphi_0],\varphi^{\ast}_0\rangle\eta,$$
which completes the proof.
\end{proof}

\section{Proof of Theorem \ref{thm:stable2}}\label{sec:4}

In this section, we prove the stability result --- Theorem \ref{thm:stable2}.
Before proceeding, we also recall some known results and notations from Kielh\"{o}fer \cite[I.12]{Kielhoefer2012} for the reader's convenience.
Our proof will be based on those results and use the same notations.

\subsection{Preliminaries}
Following \cite{Kielhoefer2012}, we denote by $\mu$ the \emph{Floquet exponents}, which are
the eigenvalues of the linearized operator $\frac{\mathrm{d}}{\mathrm{d} t}-{D}_{x} F(x(t), \lambda)$ in the space of $p$-periodic functions.
The stability is determined by the so-called \emph{Principle of Linearized Stability}:
\begin{quote}
The $p$-periodic solution $x=x(t)$ of \eqref{a1} is (linearly) \emph{stable} if the Floquet exponent $\mu=0$ is simple and if all its Floquet exponents $\mu \neq 0$ have positive real parts.
\end{quote}
As pointed out in \cite[p.83]{Kielhoefer2012}, the reason to consider only linear stability here is that we do not know a proof for the present general setting that linear stability indeed implies asymptotic stability.
But for ODEs and many parabolic PDEs, it is true that linear stability  implies asymptotic stability (cf. \cite{Henry1981,Lunardi1995}).

In \cite[I.12]{Kielhoefer2012}, the above principle is applied to the bifurcating curve $\{(x(r), \lambda(r))\}$ of $2 \pi / \kappa(r)$-periodic solutions of \eqref{a1} given by the (nondegenerate) Hopf bifurcation theorem --- Theorem \ref{thm1}. Just as in \eqref{b1}, the substitution $t / \kappa(r)$ for $t$ fixes the period to $2 \pi,$ and the stability problem converts into the study of the eigenvalues $\mu$ of
\begin{equation}\label{eq:eigenvalue}
{D}_{x} G(x(r), \kappa(r), \lambda(r)) \psi\equiv\bigg(\kappa(r) \frac{\mathrm{d}}{\mathrm{d} t}-{D}_{x} F(x(r)(t), \lambda(r))\bigg) \psi=\mu \psi.
\end{equation}
For $r=0,$ it is known in the previous section that the operator $J_{0}\equiv{D}_{x} G\left(0, \kappa_{0}, \lambda_{0}\right)=\kappa_{0} \frac{d}{d t}-A_{0}$ has a geometrically double eigenvalue $0$ with eigenvectors $\hat{v}_{1}, \hat{v}_{2}$. So this principle does not apply for $r=0$, but is proved to hold for $r \neq 0$
under the nondegeneracy (F4) and the smoothness (F1$^{\prime}$)  (cf. \cite[pp.83--84]{Kielhoefer2012}). Thus, the stability result --- Theorem \ref{thm:stable1} is successfully established.

We also use this principle to prove the stability properties for the bifurcating solutions of the degenerate problem in the next subsection. We now recall some known results from \cite[I.12]{Kielhoefer2012}, and notice that those results do not require condition (F4)  and hence hold not only for the bifurcating solutions obtained by Theorem \ref{thm1} but also for those by Theorem \ref{thm2}. 
Similar as in the previous section, $F \in C^{4}(U \times V, Z)$ by (F1${}^{\prime}$) implies that $G \in C^{3}(\tilde{U} \times \tilde{V}, W)$.
 
Since
\begin{equation*}
  \text{$J_{0}\psi=\mu \psi, \psi(0)=\psi(2 \pi) \Leftrightarrow$
$in\kappa_{0}-\mu$ is an eigenvalue of $A_{0}$ for some $n \in \mathbb{Z}$,}
\end{equation*}
which, under condition (F7), implies  that
$\rea  \mu>0$ for all Floquet exponents $\mu \neq 0$ of
$J_{0}$.
Therefore, the (linear) stability of the bifurcating curve $\{(x(r), \lambda(r))\}$ of $2 \pi / \kappa(r)$-periodic solutions of \eqref{a1} is determined by the sign of the real part of the perturbed critical eigenvalues $\mu(r)$ of ${D}_{x} G(x(r), \kappa(r), \lambda(r))$ near $\mu(0)=0,$ at least for small $r \in(-\delta, \delta) .$

 Since the algebraic multiplicity is preserved under a small perturbation, there exist two perturbed eigenvalues $\mu_{1}(r), \mu_{2}(r)$ such that $\mu_{1}(0)=\mu_{2}(0)=0$.
On one hand, according to \cite[(I.12.7) and Proposition I.18.3]{Kielhoefer2012}, the periodic solution $x(r)$ possesses the trivial Floquet exponent $\mu_{1}(r) \equiv 0$ with a curve of eigenfunctions
$\left\{\psi_{1}(r) \mid r \in(-\delta, \delta)\right\} \subset E \cap Y$ that is twice
continuously differentiable such that
\begin{equation}\label{eq:firsteigen}
P \psi_{1}(r)=i(\psi_{0}-\overline{\psi}_{0})\quad \text{and}\quad w_{1}(r) \equiv(I-P) \psi_{1}(r)\:
\hbox{ satisfying }\: w_{1}(0)=0.
\end{equation}
Set the projections (cf. \eqref{eq:vv})
\begin{align*}
&Q_{j} z=\frac{1}{2 \pi} \Big(\int_{0}^{2 \pi}\left\langle z, \hat{v}_{j}^{\prime}\right\rangle d t \Big) \hat{v}_{j}, \quad j=1,2, \; z \in W,\qquad \left.Q_{j}\right|_{E \cap Y}=P_{j}.
\end{align*}
Then by \eqref{eq:Q}, $Q=Q_{1}+Q_{2}$, $Q_{1} Q_{2}=Q_{2} Q_{1}=0,$ both $Q_{1}$ and $Q_{2}$ are real for real $z,$ and $\psi_{1}(0)=\hat{v}_{1}, Q_{2} w_{1}(r)=0$ for all $r \in(-\delta, \delta)$.
On the other hand, there exists a linearly independent (possibly generalized) eigenfunction $\psi_{2}(r)$ with the second perturbed eigenvalue $\mu_{2}(r)$ such that $\mu_{2}(0)=0$. Precisely, according to \cite[Proposition I.12.1]{Kielhoefer2012},
there is a unique twice continuously differentiable curve $\left\{\mu_{2}(r) \mid r \in(-\delta, \delta), \mu_{2}(0)=0\right\}$ in $\mathbb{R}$ such that
\begin{equation}\label{c1}
 {D}_{x} G(x(r), \kappa(r), \lambda(r))[\hat{v}_{2}+w_{2}(r)]=\mu_{2}(r)\left(\hat{v}_{2}+w_{2}(r)\right)+\nu(r) \psi_{1}(r),
\end{equation}
where $\left\{w_{2}(r) \mid r \in(-\delta, \delta), w_{2}(0)=0\right\} \subset(I-P)(E \cap Y)$,
$\{\nu(r) \mid r \in(-\delta, \delta), \nu(0)=0\} \subset \mathbb{R}$ are also twice continuously differentiable, and $ \psi_{1}(r)=\hat v_1+ w_1(r)$.

Therefore, the first Floquet exponent $\mu_1(r)$ is trivial and simple. It follows that
under condition (F7), the stability properties of a periodic solution $x(r)(t)$ are determined by the sign of the second (nontrivial) Floquet exponent $\mu_2(r)$. That is, $x(r)$ is \emph{stable} if $\mu_2(r)>0$ and  \emph{unstable} if $\mu_2(r)<0$.

\subsection{Proofs}
\begin{proof}[\textbf{Proofs of Theorem \ref{thm:stable2} and Corollary \ref{cor:h22}. }]
Differentiating \eqref{c1} with respect to $r$ at $r=0$ gives
\begin{align*}
D_{xx}^2G(0,\kappa_0,\lambda_0)[\hat{v}_2,\hat{v}_2]+\dot{\kappa}(0)D_{x\kappa}^2G(0,\kappa_0,\lambda_0)[\hat{v}_2]
&+\dot{\lambda}(0)D_{x\lambda}^2G(0,\kappa_0,\lambda_0)[\hat{v}_2]+J_0[\dot{w}_2(0)]\\
&=\dot{\mu}_2(0)\hat{v}_2+\dot{\nu}(0)\hat{v}_1.
\end{align*}
That is,
\begin{equation}\label{c1.1}
-D_{xx}^2F(0,\lambda_0)[\hat{v}_2,\hat{v}_2]+\dot{\kappa}(0)\frac{\mathrm{d}\hat{v}_2}{\mathrm{d} t}-\dot{\lambda}(0)D_{x\lambda}^2F(0,\lambda_0)[\hat{v}_2]+J_0[\dot{w}_2(0)]=\dot{\mu}_2(0)\hat{v}_2+\dot{\nu}(0)\hat{v}_1.
\end{equation}
Here and in what follows, we usually use notations $^{\prime}=\frac{\mathrm{d}}{\mathrm{d}\lambda}$
and  $\dot{}=\frac{\mathrm{d}}{\mathrm{d}r}$ for simplicity.

By the relation $\frac{\mathrm{d}\hat{v}_2}{\mathrm{d} t}=\hat{v}_1$ and \eqref{eq:equality}, acting $\frac{1}{2\pi}\int_0^{2\pi}\langle\cdot,\hat v_2^{\ast}\rangle  \,\mathrm{d} t$ on \eqref{c1.1} yields
\begin{equation}
\dot{\mu}_2(0)=-\rea \langle D_{x\lambda}^2F(0,\lambda_0)[\varphi_0],\varphi_0^* \rangle\dot{\lambda}(0), \label{c1.2}
\end{equation}
which, together with the degeneracy (F4$^\prime$) or \eqref{t1},  gives
\begin{equation}
\dot{\mu}_2(0)=0.\label{c1.3}
\end{equation}

We next prove that
\begin{equation}\label{c1.4}
\ddot{\mu}_2(0)=\rea\frac{\mathrm{d}^2}{\mathrm{d} r^2}{D}_{r}\hat{\Phi}(r,\kappa(r),\lambda(r))\big|_{r=0}.
\end{equation}
Note that as is shown below, proving \eqref{c1.4} actually does not require
\eqref{c1.3}, (F4) or (F4$^\prime$).

On one hand, acting $\frac{1}{2\pi}\int_0^{2\pi}\langle\cdot,\hat v_2^\ast\rangle  \mathrm{d} t$ on \eqref{c1} gives
\begin{equation*} 
\mu_2(r)=\frac{1}{2\pi}\int_0^{2\pi}\langle D_xG\big(x(r),\kappa(r),\lambda(r)\big)[\hat v_2+ w_2(r)],\hat{v}_2^\ast\rangle  \mathrm{d} t.
\end{equation*} 
Differentiating $\mu_2$ with respect to $r$ once and twice yield
\begin{align} \nonumber 
\dot{\mu}_2(r)&=\frac{1}{2\pi}\int_0^{2\pi}\langle \frac{\mathrm{d}}{\mathrm{d}r}\big(D_xG\big)[\hat v_2+ w_2(r)]+D_xG[\frac{\mathrm{d}}{\mathrm{d}r} w_2(r)],\hat{v}_2^\ast\rangle  \mathrm{d} t,\\
\ddot{\mu}_2(r)&=\frac{1}{2\pi}\int_0^{2\pi}\langle\frac{\mathrm{d}^2}{\mathrm{d}r^2}\big(D_xG\big)[\hat v_2+ w_2(r)]+2\frac{\mathrm{d}}{\mathrm{d}r}\big(D_xG\big)[\frac{\mathrm{d}}{\mathrm{d}r} w_2(r)]+D_xG[\frac{\mathrm{d}^2}{\mathrm{d}r^2} w_2(r)],\hat v_2^\ast\rangle  \mathrm{d} t.\label{c.14}
\end{align}
On the other hand, let $(x(r),\kappa(r),\lambda(r))$ be the given solution and
write $x(r)=r\hat{v}_2+\psi(r\hat{v}_2,\kappa(r),\lambda(r))$.
Then we obtain from \eqref{eq:vv} and \eqref{eq:psifirst} that
\begin{equation*}
\rea{D}_{r}\hat{\Phi}(r,\kappa(r),\lambda(r))=\frac{1}{2\pi}\int_0^{2\pi}\langle D_xG\big(x(r),\kappa(r),\lambda(r)\big)[\hat v_2+D_v\psi[\hat v_2]],\hat{v}_2^\ast\rangle  \mathrm{d} t.\label{c.9}
\end{equation*}
Differentiating it with respect to $r$ once and twice yield
\begin{align}\nonumber
\rea\frac{\mathrm{d}}{\mathrm{d}r}{D}_{r}\hat{\Phi}(r,\kappa(r),\lambda(r))&=\frac{1}{2\pi}\int_0^{2\pi}\langle \frac{\mathrm{d}}{\mathrm{d}r}\big(D_xG\big)[\hat v_2+D_v\psi[\hat v_2]]+D_xG[\frac{\mathrm{d}}{\mathrm{d}r}D_v\psi[\hat v_2]],\hat{v}_2^\ast\rangle  \mathrm{d} t,\\  
\rea\frac{\mathrm{d}^2}{\mathrm{d}r^2}{D}_{r}\hat{\Phi}(r,\kappa(r),\lambda(r))&=\frac{1}{2\pi}\int_0^{2\pi}\langle\frac{\mathrm{d}^2}{\mathrm{d}r^2}\big(D_xG\big)[\hat v_2+D_v\psi[\hat v_2]]\nonumber\\
&\quad +2\frac{\mathrm{d}}{\mathrm{d}r}\big(D_xG\big)[\frac{\mathrm{d}}{\mathrm{d}r}D_v\psi[\hat v_2]]+D_xG[\frac{\mathrm{d}^2}{\mathrm{d}r^2}D_v\psi[\hat v_2]],\hat v_2^\ast\rangle  \mathrm{d} t. \label{c.15}
\end{align}

Evaluate \eqref{c.14} and \eqref{c.15} at $r=0$  and then compare them. Notice that $D_v\psi(0,\kappa_0,\lambda_0)=0$,  $ w_2(0)=0$, and $\frac{1}{2 \pi} \int_{0}^{2 \pi}\left\langle z, \hat{v}_{2}^{\ast}\right\rangle d t=0$ for any $z\in R(J_0)$  by \eqref{eq:equality}. It follows that in order to prove \eqref{c1.4}, it suffices to show that
\begin{equation}\label{left}
\frac{\mathrm{d}}{\mathrm{d}r} w_2(r)\big|_{r=0}=\frac{\mathrm{d}}{\mathrm{d}r}D_v\psi[\hat v_2]\big|_{r=0}.
\end{equation}
On one hand, taking $v=r\hat v_2$, $\kappa=\kappa(r)$ and $\lambda=\lambda(r)$ in \eqref{b9} and
 differentiating it with respect to $r$ at $r=0$ yield
\begin{equation}
(I-Q)\frac{\mathrm{d}}{\mathrm{d}r}(D_xG)[\hat v_2+D_v\psi[\hat v_2]]+(I-Q)D_xG[\frac{\mathrm{d}}{\mathrm{d}r}D_v\psi[\hat v_2]]=0.\label{c.5}
\end{equation}
On the other hand, differentiating \eqref{c1} with respect to $r$ and taking the projection $I-Q$ yield
\begin{align}\nonumber
(I-Q)\frac{\mathrm{d}}{\mathrm{d}r}(D_xG)[\hat v_2+ w_2(r)]&+(I-Q)D_xG[\frac{\mathrm{d}}{\mathrm{d}r} w_2(r)]\\
&=\frac{\mathrm{d}}{\mathrm{d}r}\mu_2\cdot w_2+\mu_2\frac{\mathrm{d}}{\mathrm{d}r} w_2+\frac{\mathrm{d}}{\mathrm{d}r}\nu\cdot w_1+\nu\frac{\mathrm{d}}{\mathrm{d}r} w_1.\label{c.4}
\end{align}
Evaluating \eqref{c.5} and \eqref{c.4} at $r=0$ and comparing them, we have
\begin{equation}\label{c.5.1}
J_0[\frac{\mathrm{d}}{\mathrm{d}r} w_2(r)]\big|_{r=0}=J_0[\frac{\mathrm{d}}{\mathrm{d}r}D_v\psi[\hat v_2]]\big|_{r=0},
\end{equation}
which implies \eqref{left} because both $\frac{\mathrm{d}}{\mathrm{d}r} w_2(r)$ and $\frac{\mathrm{d}}{\mathrm{d}r}D_v\psi[\hat v_2]$ belong to $(I-P)(E \cap Y)$.
So \eqref{c1.4} holds.

In view of $\dot{\mu}_2(0)=0$, to determine the stability properties of the bifurcating solutions,  we need to further calculate $\ddot{\mu}_2(0)$.
Direct differentiation yields
\begin{align}\nonumber
\frac{\mathrm{d}^2}{\mathrm{d} r^2}{D}_{r}\hat{\Phi}(r,\kappa(r),&\lambda(r))={D}_{rrr}^3\hat{\Phi}(r,\kappa(r),\lambda(r))+2{D}_{rr\kappa}^3\hat{\Phi}(r,\kappa(r),\lambda(r))\dot{\kappa}(r)\\\nonumber
+&2{D}_{rr\lambda}^3\hat{\Phi}(r,\kappa(r),\lambda(r))\dot{\lambda}(r)+{D}_{r\kappa\kappa}^3\hat{\Phi}(r,\kappa(r),\lambda(r))(\dot{\kappa}(r))^2\\\nonumber
+&2{D}_{r\kappa\lambda}^3\hat{\Phi}(r,\kappa(r),\lambda(r))\dot{\kappa}(r)\dot{\lambda}(r)+{D}_{r\kappa}^2\hat{\Phi}(r,\kappa(r),\lambda(r))\ddot{\kappa}(r)\\\label{mu2}
+&{D}_{r\lambda\lambda}^3\hat{\Phi}(r,\kappa(r),\lambda(r))(\dot{\lambda}(r))^2+{D}_{r\lambda}^2\hat{\Phi}(r,\kappa(r),\lambda(r))\ddot{\lambda}(r).
\end{align}
Evaluate this expression at $r=0$ and notice that $\dot{\lambda}(0)=\eta$ and $H_{11}+H_{22}\eta^2=0$ in {Theorem \ref{thm2}}. Then by Lemma \ref{lem:3.1}, we obtian from \eqref{c1.4} and \eqref{mu2}  that
\begin{equation}\label{c3}
\ddot{\mu}_2(0)=3H_{11}+H_{22}\eta^2=2H_{11}.
\end{equation}

We next show that the analogous relations like \eqref{c1.3} and \eqref{c1.4}  also hold for the trivial solution of \eqref{a1}.
Indeed, take $2\pi/\kappa_0$ as the period of the trivial solution $(0,\lambda)$, 
i.e.,  $G(0,\kappa_0,\lambda)\equiv 0$. 
It follows from \eqref{eq:perteigen} that 
$\mu(\lambda)$ and $\varphi(\lambda)$ also satisfy
\begin{equation}
{D}_xF(0,\lambda)\overline{\varphi}(\lambda)=\overline{\mu}(\lambda)\overline{\varphi}(\lambda),\quad \text{with }\overline{\mu}(\lambda_0)=-i\kappa_0\text{ and }\overline{\varphi}(\lambda_0)=\overline{\varphi}_0.
\end{equation}
Then we get
\begin{align}\nonumber
&{D}_xG(0,\kappa_0,\lambda)[\varphi(\lambda)e^{it}+\overline{\varphi}(\lambda)e^{-it}]\\\nonumber
=&(\kappa_0\frac{\mathrm{d}}{ \mathrm{d} t}-D_{x}F (0,\lambda))[\varphi(\lambda)e^{it}+\overline{\varphi}(\lambda)e^{-it}]\\\nonumber
=&i\kappa_0\varphi(\lambda)e^{it}-i\kappa_0\overline{\varphi}(\lambda)e^{-it}-D_{x}F (0,\lambda)\varphi(\lambda)e^{it}-D_{x}F (0,\lambda)\overline{\varphi}(\lambda)e^{-it}\\\nonumber
=&i\kappa_0\varphi(\lambda)e^{it}-i\kappa_0\overline{\varphi}(\lambda)e^{-it}-\mu(\lambda)\varphi(\lambda)e^{it}-\overline{\mu}(\lambda)\overline{\varphi}(\lambda)e^{-it}\\
=&-\rea \mu(\lambda)\left(\varphi(\lambda)e^{it}+\overline{\varphi}(\lambda)e^{-it}\right)+(\kappa_0-\im \mu(\lambda))\left(i\varphi(\lambda)e^{it}-i\overline{\varphi}(\lambda)e^{-it}\right).\label{c17}
\end{align}
By the Implicit Function Theorem (cf. \cite[Proposition I.7.2]{Kielhoefer2012}), we have  $$\varphi(\lambda)=\varphi_0+\varphi_1(\lambda),$$
 where $\varphi_1$ satisfies $\langle \varphi_1(\lambda),\varphi^{\ast}_0\rangle=0$.
Rewrite $\varphi(\lambda)e^{it}+\overline{\varphi}(\lambda)e^{-it}$ and $i\varphi(\lambda)e^{it}-i\overline{\varphi}(\lambda)e^{-it}$ in the form 
\begin{equation}\label{eq:decomphi}
 \varphi(\lambda)e^{it}+\overline{\varphi}(\lambda)e^{-it}=\hat v_2+\hat w_2(\lambda),\quad
i\varphi(\lambda)e^{it}-i\overline{\varphi}(\lambda)e^{-it}=\hat v_1+\hat w_1(\lambda),
\end{equation}
in which $\hat v_1$ and $\hat v_2$ are defined as in \eqref{eq:vv}.
Inserting \eqref{eq:decomphi} into \eqref{c17} yields 
\begin{equation}
{D}_xG(0,\kappa_0,\lambda)[\hat v_2+\hat w_2(\lambda)]=-\rea \mu(\lambda)(\hat v_2+\hat w_2(\lambda))+(\kappa_0-\im \mu(\lambda))(\hat v_1+\hat w_1(\lambda)).\label{c18}
\end{equation}
Since $-\rea(\lambda_0)=\kappa_0-\im\mu(\lambda_0)=0$ and $\hat w_1(\lambda_0)=\hat w_2(\lambda_0)=0$, similar to the proofs of \eqref{c1.3} and \eqref{c1.4}, we obtain from \eqref{c18} that 
\begin{equation}\label{c1.5}
-\rea\mu^{\prime}(\lambda_0)=0,\quad -\rea\mu^{\prime\prime}(\lambda_0)=\rea\frac{\mathrm{d}^2}{\mathrm{d} \lambda^2}{D}_{r}\hat{\Phi}(0,\kappa_0,\lambda)\big|_{\lambda=\lambda_0}.
\end{equation}
Lemma \ref{lem:3.1} implies that
\begin{equation}\label{c4}
-\rea\mu^{\prime\prime}(\lambda_0)=\rea{D}_{r\lambda\lambda}^3\hat{\Phi}(0,\kappa_0,\lambda_0)=H_{22},
\end{equation}
which completes the proof of Corollary \ref{cor:h22}. Notice that this part does not need condition (F7). 

Therefore, it follows from \eqref{c3} and \eqref{c4} that the signs of $H_{11}$ and $H_{22}$ decide the stability properties of the bifurcating solutions and the trivial solution  near $\lambda_0$, respectively.
Under the assumptions of Theorem \ref{thm2}(2),
since $\det H_0=H_{11} H_{22}<0$, if $H_{22}>0$, then $H_{11}<0$ and hence $\mu_2(r)<0$ near $(0,\lambda_0)$, which implies that the bifurcating periodic solution $\{(x(r),\lambda(r))\}$ is unstable. Similarly, if $H_{22}<0$, then  $H_{11}>0$ and the stability becomes reversed.
It also follows that the bifurcating periodic solutions and the trivial solution have reversed stability,
and hence the principle of exchange of stability also holds.
\end{proof}


\section{Applications}\label{sec:4}

Theorems \ref{thm2} and \ref{thm:stable2} have a wide range of applications, especially to multi-parameter problems. For example,
by rechecking some known bifurcation problems but in the range of parameters where the nondegeneracy condition (F4) fails, it is possible to find new or hidden bifurcating branches. We next present an example by revisiting a bifurcation problem of PDE system which had been deeply studied earlier in Yi et al. \cite{Yi2009}.

Consider the following diffusive predator--prey system with Holling type-II nonlinearity
\begin{equation}\label{eq:final}
\begin{cases}
u_t-d_1u_{xx}=u\big(1-\frac{u}{k}\big)-\frac{muv}{1+u},&x\in(0,\ell\pi),\; t>0,\\
v_t-d_2v_{xx}=-\theta v+\frac{muv}{1+u},&x\in(0,\ell\pi),\; t>0,\\
u_x(0,t)=v_x(0,t)=0,\quad u_x(\ell\pi,t)=v_x(\ell\pi,t)=0,&t>0,\\
u(x,0)=u_0(x)\geqslant 0,\quad v(x,0)=v_0(x)\geqslant 0,
\end{cases}
\end{equation}
where $d_1$, $d_2$, $\ell$, $m$, $k$, $\theta\in\mathbb{R}^+$.
It is clear that the system has three constant equilibrium solutions:
$(0,0)$, $(k,0)$ and $(\lambda,v_\lambda)$,
in which $\lambda=\frac{\theta}{m-\theta}>0$ and $v_\lambda=\frac{(k-\lambda)(1+\lambda)}{km}>0$
provided $m > (1+\frac{1}{k})\theta$ (or equivalently, $0<\lambda<k$).

To investigate possible bifurcating solutions from the positive coexistence equilibrium  $(u,v)=(\lambda,v_\lambda)$,  $\lambda$ is treated as a bifurcation parameter.
From the known results of stability (cf. \cite[pp.1954--1955]{Yi2009}), it follows that each potential bifurcation point has to fall in
the interval $(0,\frac{k-1}{2}]$ with $k>1$.
To be convenient, \eqref{eq:final} is translated into the system
\begin{equation}\label{eq:final2}
\begin{cases}
u_t-d_1u_{xx}=(u+\lambda)\big(1-\frac{u+\lambda}{k}\big)-\frac{\theta(\lambda+1)(u+\lambda)(v+v_\lambda)}{\lambda(1+u+\lambda)},&x\in(0,\ell\pi),\; t>0,\\
v_t-d_2v_{xx}=-\theta(v+v_\lambda)+\frac{\theta(\lambda+1)(u+\lambda)(v+v_\lambda)}{\lambda(1+u+\lambda)},&x\in(0,\ell\pi),\; t>0,\\
u_x(0,t)=v_x(0,t)=0,\quad u_x(\ell\pi,t)=v_x(\ell\pi,t)=0,&t>0.
\end{cases}
\end{equation}
Thus, 
it is equivalent to studying Hopf bifurcation of \eqref{eq:final2} from the trivial solution $(u,v)=(0,0)$ for $\lambda\in(0,\frac{k-1}{2}]$ and $k>1$.

Rewrite the equations of \eqref{eq:final2} in the abstract form
\begin{equation}\label{eq:final3}
\frac{\mathrm{d}\mathcal{U}}{\mathrm{d}t}=F(\mathcal{U},\lambda),\qquad \text{where }\; \mathcal{U}=\begin{pmatrix}
u\\v
\end{pmatrix},\;
F(\mathcal{U},\lambda)=\begin{pmatrix}
d_1u_{xx}+f(u,v)\\d_2v_{xx}+g(u,v)
\end{pmatrix},
\end{equation}
with $f(u,v)\!=\!(u+\lambda)(1-\frac{u+\lambda}{k})-\frac{\theta(\lambda+1)(u+\lambda)(v+v_\lambda)}{\lambda(1+u+\lambda)}$ and $g(u,v)\!=\!-\theta(v+v_\lambda)+\frac{\theta(\lambda+1)(u+\lambda)(v+v_\lambda)}{\lambda(1+u+\lambda)}$.
Choose the suitable spaces
\begin{equation}\label{eq:spaces}
X=\left\{(u,v)\in H^2(0,\ell\pi)\times H^2(0,\ell\pi)\big|(u_x,v_x)|_{x=0,\ell\pi}=0\right\},\quad Z=L^2(0,\ell\pi)\times L^2(0,\ell\pi),
\end{equation}
equipped with the standard complex inner product,
and consider the linearized operator at the trivial solution
\begin{equation}\label{eq:operator}
D_\mathcal{U}F(0,\lambda)=
\begin{pmatrix}
d_1\frac{\partial^2}{\partial x^2}+A(\lambda)&-\theta\\
\frac{k-\lambda}{k(1+\lambda)}&d_2\frac{\partial^2}{\partial x^2}
\end{pmatrix}
\quad \hbox{ with } A(\lambda)=\frac{\lambda(k-1-2\lambda)}{k(1+\lambda)}.
\end{equation}
In view of the basis $\{\cos \frac{nx}{\ell}\}_{n\in \mathbb{N}}$ of $L^2(0,\ell\pi)$ and the linear operator $
\begin{pmatrix}
-\frac{d_1 n^2}{\ell^2}+A(\lambda) &-\theta\\
\frac{k-\lambda}{k(1+\lambda)}&-\frac{d_2 n^2}{\ell^2}
\end{pmatrix},$
solving the corresponding characteristic equation
\begin{equation}\label{eq:charaeq}
\beta^{2}-\beta T_{n}(\lambda)+D_{n}(\lambda)=0,\quad n\in \mathbb{N},
\end{equation}
in which $
T_{n}(\lambda)=A(\lambda)-\frac{\left(d_{1}+d_{2}\right) n^{2}}{\ell^{2}}
$ and $D_n(\lambda)=\frac{\theta(k-\lambda)}{k(1+\lambda)}
-A(\lambda)\frac{d_2n^2}{\ell^2}+\frac{d_1d_2n^4}{\ell^4}$,
gives the eigenvalues $\mu_n(\lambda)$ of $D_\mathcal{U}F(0,\lambda)$ over the complexified space $X_{c}$ of $X$ having the form
\begin{equation}\label{eq:eigen}
\mu_n(\lambda)=\alpha_n(\lambda)\pm i\omega_n(\lambda),\quad n\in \mathbb{N},
\end{equation}
with
\begin{equation*} 
\alpha_n(\lambda)=\frac{A(\lambda)}{2}-\frac{(d_1+d_2)n^2}{2\ell^2},\quad
\omega_n(\lambda)=\sqrt{D_n(\lambda)-\alpha^2_n(\lambda)}.
\end{equation*}
Differentiating $\alpha_n$ with respect to $\lambda$ yields
\begin{equation}\label{eq:d7}
\alpha_n^{\prime}(\lambda)=\frac{A^{\prime}(\lambda)}{2}=\frac{k-1-4\lambda-2\lambda^2}{2k(1+\lambda)^2},\quad
\alpha_n^{\prime\prime}(\lambda)=\frac{A^{\prime\prime}(\lambda)}{2}=-\frac{k+1}{k(1+\lambda)^3}.
\end{equation}
Since $\alpha_n^{\prime}$ and $\alpha_n^{\prime\prime}$ are independent of $n$, we omit the subscript $n$ in what follows for simplicity.

 Set
\begin{equation}\label{d7.1}
\begin{split}
    & \lambda_0^H=\frac{k-1}{2},\quad\lambda_* =\sqrt\frac{1+k}{2}-1\in \big(0,\lambda_0^H\big), \quad \ell_n= n\sqrt\frac{d_1+d_2}{M_*},
\end{split}
\end{equation}
with $M_*=A(\lambda_*)=\frac{(\sqrt{k+1}-\sqrt{2})^2}{k}$.
Restricted to $(0,\lambda^H_0]$, then $\lambda_0^H$ and $\lambda_*$ uniquely solve $\alpha_0(\lambda)=0$ and $\alpha^{\prime}(\lambda)=0$, respectively.
In addition, when $\ell=\ell_n$,  $\lambda_*$ uniquely solves $\alpha_n(\lambda)=0$ for each $n\geqslant 1$; when $\ell\in \left(\ell_{n}, \ell_{n+1}\right]$, $\alpha_j(\lambda)$ admits exactly two roots $\lambda_{j, \pm}^{H}(\ell)$ for each $1 \leqslant j \leqslant n$.
It is well known that for each $\ell > 0$, $\lambda_0^H$ is the unique Hopf bifurcation point of the corresponding ODE system (i.e., $d_1=d_2=0$ in \eqref{eq:final}).

Yi et al. \cite[Theorem 2.4]{Yi2009} established the following result of Hopf bifurcation.
\begin{thm}[\cite{Yi2009}] \label{thm:shi}
 Suppose that the constants $d_{1}, d_{2}, \theta>0$, and $k>1$ satisfy
\begin{equation}\label{eq:olddd}
\frac{d_{1}}{d_{2}}>\frac{(\sqrt{k+1}-\sqrt{2})^{4}}{4 \theta k}.
\end{equation}
 Then for any $\ell$ in $\left(\ell_{n}, \ell_{n+1}\right]$, there exist $2n$ points $\lambda_{j, \pm}^{H}(\ell), 1 \leqslant j \leqslant n$, satisfying
$$
0<\lambda_{1,-}^{H}(\ell)<\lambda_{2,-}^{H}(\ell)<\cdots<\lambda_{n,-}^{H}(\ell)<\lambda_{*}
<\lambda_{n,+}^{H}(\ell)<\cdots<\lambda_{2,+}^{H}(\ell)<\lambda_{1,+}^{H}(\ell)<\lambda_{0}^{H},
$$
such that the system \eqref{eq:final} undergoes a Hopf bifurcation at $\lambda=\lambda_{j, \pm}^{H}$ or $\lambda=\lambda_{0}^{H}$, and the bifurcating periodic solutions can be parameterized in the form
\begin{equation}\label{eq:exform}
\left\{\begin{array}{l}
u(s)(x, t)=\lambda+s\left(a_{n} e^{2 \pi i t / T(s)}+\overline{a_{n}} e^{-2 \pi i t / T(s)}\right) \cos \frac{nx}{\ell} +o(s^{2}), \\
v(s)(x, t)=v_\lambda+ s\left(b_{n} e^{2 \pi i t / T(s)}+\overline{b_{n}} e^{-2 \pi i t / T(s)}\right) \cos \frac{nx}{\ell} +o(s^{2}),
\end{array}\right.
\end{equation}
where
\begin{equation*}
T(s)=\frac{2 \pi}{\omega_0}(1+\tau_{2} s^{2})+o(s^{4}), \quad \text{with } \tau_{2}=-\frac{1}{\omega_0}\big[\im (c_{1}(\lambda_{0}))-\frac{ \omega^{\prime}(\lambda_{0})}{\alpha^{\prime}(\lambda_{0})}\rea (c_{1}(\lambda_{0}))\big]
\end{equation*}
(here, $c_1(\lambda)$ is coefficient of $z\bar{z}$ in the Poincar\'{e} normal form, $\omega_0=\omega_{j}(\lambda_0)$, and $\lambda_0=\lambda_{j, \pm}^{H}$ or $\lambda_{0}^{H}$).
Moreover,
\begin{enumerate}
  \item[(1)]
  The bifurcating periodic solutions from $\lambda=\lambda_{0}^{H}$ are spatially homogeneous, which coincides with the periodic solution of the corresponding ODE system;
  \item[(2)]
  The bifurcating periodic solutions from $\lambda=\lambda_{j, \pm}^{H}$ are spatially non-homogeneous.
\end{enumerate}
\end{thm}

The interesting work of Yi et al. \cite{Yi2009} had shown the complex spatiotemporal dynamics and attracted considerable attention.
However for problem \eqref{eq:final}, it is still unclear whether a Hopf bifurcation occurs at the point $\lambda=\lambda_*$ for some $\ell_n$ or not. Since $\alpha^{\prime}(\lambda_*)= 0$, Theorem \ref{thm1} no longer applies to this case.
By applying Theorem \ref{thm2}, we next give an answer to this question.
Actually, we prove that a degenerate Hopf bifurcation is able to occur at $\lambda_*$ when $\ell=\ell_{n}$ for each $n\geqslant 1$.

For problem  \eqref{eq:final3}, consider $H_{11}$ defined by \eqref{eq:HH}$_1$, in which taking $x=\mathcal{U}$, $\lambda_0=\lambda_*$ and $\ell=\ell_n$. Then
\begin{equation}\label{eq:H11}
\begin{split}
H_{11}=&\rea \big\langle-F_{\mathcal{UUU}}(0,\lambda_*)[\varphi_0,\varphi_0,\overline{\varphi}_0]
+2F_{\mathcal{UU}}(0,\lambda_*)\big[\varphi_0,A_0^{-1}F_{\mathcal{UU}}(0,\lambda_*)[\varphi_0,\overline{\varphi}_0]\big]\\
& \qquad\qquad-F_{\mathcal{UU}}(0,\lambda_*)\big[\overline{\varphi}_0,(2i\omega_0 I-A_0)^{-1}F_{\mathcal{UU}}(0,\lambda_*)[\varphi_0,\varphi_0]\big], \varphi_0^*\big\rangle.\\
\end{split}
\end{equation}
where
\begin{equation}\label{eq:notations}
\begin{split}
A_0&=D_\mathcal{U}F(0,\lambda_*)=
\begin{pmatrix}
d_1\frac{\partial^2}{\partial x^2}+M_*&-\theta\\
\frac{\sqrt{2(k+1)}-1}{k}&d_2\frac{\partial^2}{\partial x^2}\end{pmatrix},\\
\varphi_0&=\begin{pmatrix}
1\\\frac{d_2 n^2}{\theta \ell^2_n}-i\frac{ \omega_0}{\theta}
\end{pmatrix}\cos\frac{nx}{\ell_n}, \quad \varphi^{*}_0=\frac{1}{\ell_n\pi}\begin{pmatrix}
1+i\frac{d_2n^2}{\omega_0\ell^2_n}\\
-i\frac{\theta}{\omega_0}
\end{pmatrix}\cos\frac{nx}{\ell_n},\\
\omega_0&=\omega_n(\lambda_*)=\bigg[\frac{\theta(\sqrt{2(k+1)}-1)}{k}
-\frac{d_2^2M_*^2}{(d_1+d_2)^2}\bigg]^{1/2}.
\end{split}
\end{equation}
$H_{11}$ is actually independent of $n$ (to be proved later).

The next theorem is our main result in this section and it provides a supplement for Theorem \ref{thm:shi}.
\begin{thm}\label{thm:final}
 Suppose that the constants $d_{1}, d_{2}, \theta>0$, and $k>1$ satisfy
\begin{equation}\label{eq:newdd}
\frac{d_{1}}{d_{2}}>\frac{(\sqrt{k+1}-\sqrt{2})^{4}}{4 \theta k (\sqrt{2(k+1)}-1)}.
\end{equation}
Suppose that $H_{11}\neq 0$. Then for any $\ell=\ell_{n}$ $(n\geqslant 1)$, the following assertions hold:
\begin{enumerate}
  \item[(i)] If $H_{11}>0$, then no Hopf bifurcation occurs at $\lambda=\lambda_*$, i.e.,
  the set of periodic solutions of \eqref{eq:final} near $\lambda=\lambda_*$ consists of only the constant equilibrium solution $(u,v)=(\lambda,v_\lambda)$.
  \item[(ii)] If $H_{11}<0$, then system \eqref{eq:final} exhibits a (degenerate) Hopf bifurcation at $\lambda=\lambda_*$, i.e.,
apart from the constant equilibrium solution $(\lambda,v_\lambda)$,
the set of the periodic solutions of \eqref{eq:final} near $\lambda=\lambda_*$ contains a continuously differentiable curve $\{(u(r),v(r))^T, \lambda(r))\}$ of $2 \pi / \omega(r)$-period through $((u(0),v(0))^T, \lambda(0))=\left((\lambda_*,v_{\lambda_*})^T, \lambda_*\right)$ with $\omega(0)=\omega_0$,  satisfying
\begin{equation}\label{eq:tangent2}
\big((\dot{u}(0),\dot{v}(0))^T,\dot{\lambda}(0),\dot{\omega}(0)\big)=\big(2\rea(\varphi_0 e^{i\omega_0 t}),\eta,-\frac{\theta}{k\omega_0}\eta\big),
\end{equation}
with $\eta=2^{-\frac{3}{4}}k^{\frac{1}{2}}(k+1)^{\frac{1}{4}}(-H_{11})^{\frac{1}{2}}$.
Every other periodic solution of \eqref{eq:final} in a neighborhood of $\left((\lambda_*,v_{\lambda_*})^T, \lambda_*\right)$ is obtained from $\{(u(r),v(r))^T, \lambda(r))\}$ by a phase shift
$$S_{\theta} (u(r), v(r))^T=(u(r)(t+\theta), v(r)(t+\theta))^T.$$
\end{enumerate}
\end{thm}

\begin{rem} (a) In Theorem \ref{thm:shi}, the Hopf bifurcations obtained at $\lambda_{j, \pm}^{H}$ and $\lambda_0^H$ are
subcritical or supercritical. In contrast, the Hopf bifurcation at $\lambda_*$ obtained by Theorem \ref{thm:final}
is transcritical because of \eqref{eq:tangent2}.\\
 (b) Except for $\lambda^H_0$,  bifurcating solutions from the other bifurcation points $\lambda_0\in (0,\lambda^H_0)$ (i.e., $\lambda_0=\lambda_{j, \pm}^{H}$ or $\lambda_*$) are always unstable. In fact,
 for each $\ell>0$ and  $1\leqslant j\leqslant n$,
  $$\alpha_0(\lambda_0)=\frac{A(\lambda_0)}{2}
  >\frac{A(\lambda_0)}{2}-\frac{(d_1+d_2)j^2}{2\ell^2}=\alpha_j(\lambda_0),
  $$
   which implies that in view of \eqref{eq:eigen}, there are always eigenvalues of $A_{0}$ in the right complex half-plane, because $\alpha_0(\lambda_0)>0$, $\alpha_j(\lambda_{j,\pm}^H)=0$, and $\alpha_n(\lambda_*)=0$.
\end{rem}

\begin{proof}[\textbf{Proof of Theorem \ref{thm:final}.}]
Equivalently, consider Hopf bifurcation for system \eqref{eq:final3} from the trivial solution $(u,v)=(0,0)$ at $\lambda=\lambda_*\in (0,\lambda^H_0)$.
By \eqref{eq:final3}--\eqref{eq:operator}, it is clear that (F1$^{\prime}$) and (F2) hold.

When $\ell=\ell_n$,
according to the notations in \eqref{eq:notations},
a direct computing implies that $N(i\omega_0I-A_0)=\spa\{\varphi_0\}$ and
 $N(i\omega_0I-A_0)\oplus R(i\omega_0I-A_0)=X$.
Thus $ i \omega_0$  is a simply eigenvalue of $A_0$ and is of Fredhlom of index zero. So is $-i \omega_0$ with $\overline{\varphi}_0$ instead of $\varphi_0$, and hence (F3) is satisfied.

Since $\alpha^{\prime}(\lambda_{*}) = 0$, (F4$^{\prime}$) holds.
Notice that for each $n\geqslant 1$,  $\lambda_*$ solves $\alpha_n(\lambda)=0$ if and only if $\ell=\ell_n$.
Taking $\lambda=\lambda_*$ and $\ell=\ell_n$ in \eqref{eq:charaeq} imply that
$T_{n}(\lambda_*)=2\alpha_n(\lambda_*)=0$ and $ T_{j}(\lambda_*) \neq 0$ for $j \neq n$. Moreover, since \eqref{eq:newdd} holds, it follows that
$$ D_j(\lambda_*)=\frac{\theta(\sqrt{2(k+1)}-1)}{k}
-\frac{(\sqrt{k+1}-\sqrt{2})^2}{ k}\frac{d_2j^2}{\ell_n^2}+\frac{d_1d_2j^4}{\ell_n^4}>0\quad \text{for all } j>0.$$
So $\pm i\omega_0$ is the unique pair of complex eigenvalues on the imaginary axis and hence (F5) holds.
Condition (F6), i.e. generating of the analytic semigroup, follows from  \cite[pp.224--233]{Hassard1981}.

Next, we claim that $H_{11}$ is independent of $n$.
In fact, since $\frac{\ell_n}{n}= \sqrt\frac{d_1+d_2}{M_*}$,
it follows from \eqref{eq:notations} that $\varphi_0$ does not depend on $n$ but $\varphi^{*}_0$ contains a factor $\frac{1}{\ell_n\pi}$.
However, in view of \eqref{eq:H11} and \eqref{eq:notations}, we have
\begin{align}\nonumber
H_{11}&=\rea\left\langle (\cdots),\varphi^{\ast}_0\right\rangle
= \rea\int^{\ell_n\pi}_0 (\cdots) \overline{\varphi^{*}_0}\,\mathrm{d}x\\\label{eq:h11cos}
&=\frac{1}{\ell_n\pi}\rea\int^{\ell_n\pi}_0 (\cdots)\cos^2\frac{n x}{\ell_n}+(\cdots)\cos^2\frac{n x}{\ell_n} \cos \frac{2n x}{\ell_n}+(\cdots)\cos^4\frac{n x}{\ell_n}\,\mathrm{d}x,
\end{align}
where the terms in $(\cdots)$ are independent of $n$ and $x$.
Since
$$\frac{1}{\ell_n\pi}\int^{\ell_n\pi}_0 \cos^2\frac{n x}{\ell_n}\,\mathrm{d}x=\frac{1}{2}, \quad \frac{1}{\ell_n\pi}\int^{\ell_n\pi}_0 \cos^2\frac{n x}{\ell_n}\cos\frac{2n x}{\ell_n}\,\mathrm{d}x=\frac{1}{4},\quad \frac{1}{\ell_n\pi}\int^{\ell_n\pi}_0 \cos^4\frac{n x}{\ell_n}\,\mathrm{d}x=\frac{3}{8},$$
it follows from \eqref{eq:h11cos} that the claim is true.

By Corallary \ref{cor:h22}, we obtain from \eqref{eq:d7} and \eqref{d7.1} that
\begin{equation*} 
H_{22}=-\alpha''(\lambda_*)=\frac{2\sqrt{2}}{k\sqrt{k+1}}> 0.
\end{equation*}
Since $H_{11}\neq 0$, it follows that \\
(1) If $H_{11}> 0,$ then  $\det H_0=H_{11}H_{22}>0$ and hence by Theorem \ref{thm2}(1) the set of periodic solutions of \eqref{eq:final2} near $\lambda=\lambda_*$ is only the trivial solution line $((0,0),\lambda)$, which is corresponding to the constant equilibrium solution $(u,v)=(\lambda,v_\lambda)$ of \eqref{eq:final} for each $\lambda$.\\
(2) If $H_{11}< 0,$ then $\det H_0<0$ and hence by Theorem \ref{thm2}(2) a degenerate Hopf bifurcation happens at $\lambda=\lambda_*$ from the trivial solution line, with the tangent vector given by \eqref{eq:tangent1}.
By \eqref{eq:operator} and \eqref{d7.1},
a direct computation implies $
D_{\mathcal{U}\lambda}^2F(0,\lambda_*)=\begin{pmatrix}
0&0\\
-\frac{2}{k}&0
\end{pmatrix}$.
It follows from  \eqref{eq:tangent1} and \eqref{eq:notations} that
$$\dot{\omega}(0)=\im \left\langle D_{\mathcal{U}\lambda}^2F(0,\lambda_*)[\varphi_0],\varphi^{\ast}_0\right\rangle\eta
=-\frac{\theta}{k\omega_0}\eta.$$
So \eqref{eq:tangent2} holds.
\end{proof}

At final, we conclude this section with an example.

\begin{example}
We now use Theorem \ref{thm:final} to analyze two specific cases of the multi-parameter problem \eqref{eq:final}.
\begin{enumerate}[(1)]
\item  Recall an example that was studied earlier in \cite[Example 2.6]{Yi2009} --- the case: $d_1=1$, $d_2=3$, $k=17$, $\theta=4$ in \eqref{eq:final}.
    It was known that $\lambda_{*}=2, M_{*}=8 / 17, \ell_{n}=\sqrt{17/2} n $, and \eqref{eq:olddd} holds.
    By Theorem \ref{thm:shi}, it was shown in \cite{Yi2009} that for $\ell=2 \sqrt{119} / 7\in (\ell_{1}, \ell_{2}] $, the set of Hopf bifurcation points $\Lambda_{1}=\left\{\lambda_{1,-}^{H}, \lambda_{1,+}^{H}, \lambda_{0}^{H}\right\}=\{1,7/2,8\}$.
        We now revisit this example in the view of Theorem \ref{thm:final}.
     It is clear that \eqref{eq:newdd} holds since it is weaker than \eqref{eq:olddd}.
     By a lengthy computation, we obtain from \eqref{eq:H11} and \eqref{eq:notations} that $H_{11}(\approx 0.14597)>0$.
     Then Theorem \ref{thm:final} indicates that
     no Hopf bifurcation happens at $\lambda_*$ as
      $\ell=\ell_n$ for every $n\in\mathbb{N}_+$, which provides a supplement to the known result for the case of $\ell=\ell_n$ and $\lambda=\lambda_*$ (cf. Theorem \ref{thm:shi}).

\item Consider a fast-slow diffusive case: $d_1=48$, $d_2=1$, $k=127$, $\theta=7$, According to \eqref{d7.1}, we have $\lambda_{*}=7$, $M_{*}={98}/{127}$, $ \ell_{n}=\sqrt{127/2}n$, then \eqref{eq:olddd} and \eqref{eq:newdd} are satisfied.
    By a lengthy and tedious computation, we get that $H_{11}(\approx -0.00025)<0$.
    So Theorem \ref{thm:final} indicates that $\lambda_{*}$ is a Hopf bifurcation point as $\ell=\ell_n$ for each $n\in\mathbb{N_+}$.
    Thus, we find a hidden bifurcation point, which was unknown in the earlier work.
    In contrast, Theorem
    \ref{thm:shi} provides bifurcating information for the case of $\lambda=\lambda_{j, \pm}^{H}$ or $\lambda_{0}^{H}$ but not for $\lambda=\lambda_*$.
\end{enumerate}
\end{example}

\section*{Acknowledgments}

The first and second authors were supported in part by Guangdong Basic and Applied Basic Research Foundation of China (2022A1515011867) and (2020A1515011148), respectively, which are gratefully acknowledged.


\bibliographystyle{plainnat} 
\bibliography{hopf}
\end{document}